\documentclass[reqno,12pt]{amsart}

\usepackage[active]{srcltx}
\usepackage{color}
\usepackage{amsmath,amsthm,amssymb}
\usepackage{epsfig}
\usepackage{graphicx}
\usepackage{amssymb}
\usepackage{latexsym}
\usepackage{pdfpages}

\usepackage{enumitem}

\usepackage{ulem}

\usepackage{ifpdf}

\ifpdf 
\usepackage[hidelinks]{hyperref}
\else 
\fi

\usepackage[usenames,dvipsnames]{pstricks}
\usepackage{epsfig}
\usepackage{pst-grad} 
\usepackage{pst-plot} 

\usepackage{color}

\newtheorem{theorem}{Theorem}[section]
\newtheorem{lemma}[theorem]{Lemma}
\newtheorem{definition}[theorem]{Definition}
\newtheorem{proposition}[theorem]{Proposition}
\newtheorem{corollary}[theorem]{Corollary}
\newtheorem{remark}[theorem]{Remark}

\newtheorem*{theorem*}{\it Theorem}

\def\R{\mathbb R}
\def\N{\mathbb N}

\numberwithin{equation}{section}

\def\1{\raisebox{2pt}{\rm{$\chi$}}}

\def\Xint#1{\mathchoice
	{\XXint\displaystyle\textstyle{#1}}%
	{\XXint\textstyle\scriptstyle{#1}}%
	{\XXint\scriptstyle\scriptscriptstyle{#1}}%
	{\XXint\scriptscriptstyle\scriptscriptstyle{#1}}%
	\!\int}
\def\XXint#1#2#3{{\setbox0=\hbox{$#1{#2#3}{\int}$}
		\vcenter{\hbox{$#2#3$}}\kern-.5\wd0}}

\def\dashint{\Xint-}

\newcommand{\twopartdef}[4]
{
	\left\{
	\begin{array}{ll}
		#1 & #2 \\
		#3 & #4
	\end{array}
	\right.
}

\begin{document}
	
\title[Total variation flow in metric measure spaces]{\bf The Neumann and Dirichlet problems for the total variation flow in metric measure spaces}
	
\author[W. G\'{o}rny and J. M. Maz\'on]{Wojciech G\'{o}rny and Jos\'e M. Maz\'on}
	
\address{ W. G\'{o}rny: Faculty of Mathematics, Universit\"at Wien, Oskar-Morgerstern-Platz 1, 1090 Vienna, Austria; Faculty of Mathematics, Informatics and Mechanics, University of Warsaw, Banacha 2, 02-097 Warsaw, Poland
\hfill\break\indent
{\tt  wojciech.gorny@univie.ac.at }
}
	
\address{J. M. Maz\'{o}n: Departamento de An\`{a}lisis Matem\`atico,
Universitat de Val\`encia, Dr. Moliner 50, 46100 Burjassot, Spain.
\hfill\break\indent
{\tt mazon@uv.es }}
	
%
%
	
	
\setcounter{tocdepth}{1}

\date{\today}
	
\begin{abstract}
We study the Neumann and Dirichlet problems for the total variation flow in metric measure spaces. We prove existence and uniqueness of weak solutions and study their asymptotic behaviour. Furthermore, in the Neumann problem we provide a notion of solutions which is valid for $L^1$ initial data, as well as prove their existence and uniqueness. Our main tools are the first-order linear differential structure due to Gigli and a version of the Gauss-Green formula.
\end{abstract}
	
\maketitle

{\renewcommand\contentsname{Contents}
\setcounter{tocdepth}{3}
\tableofcontents}

\section{Introduction}

Since its introduction as a means of solving the denoising problem in the seminal work by Rudin, Osher and Fatemi (\cite{ROF}), the total variation flow in Euclidean spaces
\begin{equation}\label{eq:tvflowintro}
 u_t = \mbox{div} \bigg( \frac{Du}{|Du|} \bigg) \qquad \mbox{in } \Omega \times (0,\infty)
\end{equation}
has remained one of the most popular tools in image processing. From the mathematical point of view, the study of the total variation flow in Euclidean spaces  was done in two principal ways. The first approach, which uses as main tools the classical theory of maximal monotone operators due to Brezis (\cite{Brezis}) and the Crandall-Liggett Theorem (\cite{CrandallLiggett}),  a characterisation of solutions using vector fields with integrable divergence and a Gauss-Green formula in the form proved by Anzelotti in \cite{Anz}, was established in \cite{ABCM0} and \cite{ABCM} (see \cite{ACMBook} for a survey).  The second approach, based on the definition of pseudosolution introduced by Lichnewsky and Temam in \cite{LT},  was developed in \cite{BDM1} and its core is the concept of a {\it variational solution}.  The two approaches are complementing; the main advantage of the first one is that it enables a uniform approach to the total variation flow on the whole space and Neumann and Dirichlet boundary conditions, as well as the possibility to explicitly identify solutions and their asymptotic profiles. The main advantage of the second one is that it is also possible to study the Dirichlet problem with time-dependent boundary conditions.

The natural setting to look for solutions to \eqref{eq:tvflowintro} is to require that for a.e. $t \in (0,\infty)$ the function $u(\cdot,t)$ is a function of bounded variation, i.e. its distributional derivative is a Radon measure. One of the main problems is to properly define the right-hand side of equation \eqref{eq:tvflowintro}, taking into account that the denominator may disappear on a set of positive Lebesgue measure; actually, this phenomenon of formation of facets it a typical property of solutions to the total variation flow. In the approach presented in \cite{ACMBook}, this is done by replacing $\frac{Du}{|Du|}$ with a vector field with integrable divergence which agrees $|Du|$-a.e. with the Radon-Nikodym derivative $\frac{d Du}{d|Du|}$.

In this paper, we will study the total variation flow on bounded domains in metric measure spaces with respect to Neumann and Dirichlet boundary conditions. We assume that the metric space $(\mathbb{X},d,\nu)$ is complete, separable, the measure $\nu$ is doubling and satisfies a Poincar\'e inequality, and we require a suitable assumption on the domain $\Omega$ which guarantees existence of traces. Recently, related problems attracted considerable attention; for instance, the corresponding elliptic problem for the nonhomogeneous Neumann boundary conditions was studied in \cite{LMS}, and a nonlocal version of the total variation flow in metric random walk spaces (which includes as a particular case the total variation flow in graphs and also the nonlocal case for nonsingular kernel) was studied in \cite{MST1}.  Most importantly, the concept of variational solutions to the Dirichlet problem for the total variation flow introduced in \cite{BDM1} was generalised to metric measure spaces in \cite{BCC}. Then, the authors proved existence and uniqueness of the variational solution.

Our main goal in this paper is to provide a unified framework for the study of the total variation flow in the metric setting. It will be introduced in the spirit of the results in \cite{ACMBook}, using as a main tool the metric version of the Anzelotti-Gauss-Green formula developed in \cite{GM2021-2} (see also our previous work \cite{GM2021} about the Cauchy problem for the $p$-Laplacian operator). We will provide a characterisation of solutions and prove their existence and uniqueness, using the classical theory of maximal monotone operators when the data is in $L^2(\Omega,\nu)$ and the theory of completely accretive operators (\cite{BCr2}) for data in $L^1(\Omega,\nu)$, in the following three situations. In the first one, we study the evolution with Neumann boundary conditions with initial data in $L^2(\Omega,\nu)$. In the second one, we introduce the notion of entropy solutions and allow for initial data in $L^1(\Omega,\nu)$. In the third one, we study the evolution with Dirichlet boundary condition in $L^1(\partial\Omega,|D\chi_\Omega|_\nu)$ and initial data in $L^2(\Omega,\nu)$.

To the best of our knowledge, our existence and uniqueness result for $L^1$ initial data is the first result in this direction in metric measure spaces; an analogous result is not avalaible even in the case of the heat flow. Furthermore, to the best of our knowledge, the Neumann problem has not been studied in the literature in the metric setting. The Dirichlet problem was studied using the concept of variational solutions in \cite{BCC}. Our notion of solutions is a bit different, because we require more assumptions on the domain in order to introduce traces of BV functions, vector fields with integrable divergence and their normal traces, while the results in \cite{BCC} are introduced using an extension of $u$ to a slightly larger domain. On the other hand, our notion of solutions is easier to obtain and work with, and we show that it is consistent with the variational solutions up to the choice of the extension.

Let us shortly describe the contents of the paper. Because our goal is to study the Neumann and Dirichlet problems, we will work under assumptions that guarantee existence of traces of sufficiently regular functions. In particular, the metric space $(\mathbb{X},d)$ is complete and separable, and it is equipped with a Radon measure $\nu$ which is doubling and satisfies a Poincar\'e inequality. In Section \ref{sec:preliminaries}, we first recall the definitions of Sobolev and BV functions in such metric measure spaces, then the construction of the linear first-order differential structure due to Gigli, and finally the notion of Anzellotti pairings and the Gauss-Green formula.

In Section \ref{Neumann}, we study the Neumann problem for the total variation flow for a sufficiently regular bounded domain $\Omega$ in a metric measure space. In its first part, we introduce a notion of weak solutions with initial data in $L^2(\Omega,\nu)$, and we use the classical theory of maximal monotone operators (see \cite{Brezis}) to prove their existence and uniqueness for any initial datum in $L^2(\Omega,\nu)$. Moreover, as a consequence of recent results given in \cite{BuBu} on the gradient flow of a coercive and $1$-homogeneous convex functional defined on a Hilbert space, we show that the weak solutions of the Neumann problem reach the average of the initial data in finite time. We also obtain bounds of the extinction time and asymptotic profiles for finite extinction.

In the second part of Section \ref{Neumann}, we study the Neumann problem for the total variation flow with initial data in $L^1(\Omega,\nu)$. To this end, we introduce the notion of entropy solutions; the definition is based on a property of truncations of the solutions, analogously to the definition in the Euclidean case. We prove existence and uniqueness of entropy solutions and show that they are consistent with weak solutions if the initial datum lies in $L^2(\Omega,\nu)$.

In the final Section \ref{sec:dirichlet}, we study the Dirichlet problem for the total variation flow. As for the Neumann problem, we introduce a notion of weak solutions with initial data in $L^2(\Omega,\nu)$ and prove their existence and uniqueness. Furthermore, for the homogeneous problem, i.e. when the boundary datum is null, we also study the asymptotic behaviour of its solution, showing that we have extinction in finite time and obtain its asymptotic profiles. Finally, we show that the notion of weak solutions is consistent with the variational solutions introduced in \cite{BCC}.

\section{Preliminaries}\label{sec:preliminaries}

\subsection{Sobolev and BV spaces}
	
 Given a metric measure space $(\mathbb{X}, d, \nu)$ and $p \in [1,\infty)$, there are several possible definitions of Sobolev spaces on $\mathbb{X}$; nonetheless, on complete and separable metric spaces equipped with a doubling measure, all these definitions agree (see \cite{AGS1,DiMarinoTh}). In this paper, we work under these assumptions, and to simplify some arguments we choose to use the Newtonian approach. We follow the presentation in \cite{BB}.

\begin{definition}{\rm We say that a measure $\nu$ on a metric space $\mathbb{X}$ is {\it doubling}, if there exists a constant $C_d \geq 1$ such that following condition holds:
\begin{equation}
0 < \nu(B(x,2r)) \leq C_d \, \nu(B(x, r)) < \infty
\end{equation}
for all $x \in \mathbb{X}$ and $r > 0$. The constant $C_d$ is called the doubling constant of $\mathbb{X}$. }
\end{definition}

\begin{definition}{\rm  We say that $\mathbb{X}$ supports a {\it weak $(1,p)$-Poincar\'{e} inequality} if there exist constants $C_P > 0$ and $\lambda \geq 1$ such that for all balls $B \subset \mathbb{X}$, all measurable functions $f$ on $\mathbb{X}$ and all upper gradients $g$ of $f$,
$$\dashint_B \vert f - f_B \vert d \nu \leq C_P r \left( \dashint_{\lambda B} g^p d\nu \right)^{\frac1p},  $$
where $r$ is the radius of $B$ and
$$f_B:= \dashint_B f d\nu := \frac{1}{\nu(B)} \int_B f d\nu.$$
}
\end{definition}

We say that a Borel function $g$ is an {\it upper gradient} of a Borel function $u: \mathbb{X} \rightarrow \R$ if for all curves $\gamma: [0,l_\gamma] \rightarrow \mathbb{X}$ we have
$$ \left\vert u(\gamma(l_\gamma)) - u(\gamma(0)) \right\vert \leq \int_\gamma g \,  ds.$$
If this inequality holds for $p$-almost every curve, i.e. the $p$-modulus (see for instance \cite[Definition 1.33]{BB}) of the family of all curves for which it fails equals zero, then we say that $g$ is a {\it $p$-weak upper gradient} of $u$.

The Sobolev-Dirichlet class $D^{1,p}(\mathbb{X})$ consists of all Borel functions $u: \mathbb{X} \rightarrow \R$ for which there exists  an upper gradient (equivalently: a $p$-weak upper gradient) which lies in $L^p(\mathbb{X},\nu)$. The Sobolev space $W^{1,p}(\mathbb{X}, d, \nu)$ is defined as
$$W^{1,p}(\mathbb{X}, d, \nu):= D^{1,p}(\mathbb{X}) \cap L^p(\mathbb{X}, \nu).$$
In the literature, this space is sometimes called the Newton-Sobolev space (or Newtonian space) and is denoted $N^{1,p}(\mathbb{X})$. The space $W^{1,p}(\mathbb{X},d,\nu)$ is endowed with the norm
\begin{equation*}
\| u \|_{W^{1,p}(\mathbb{X},d,\nu)} = \bigg( \int_{\mathbb{X}} |u|^p \, d\nu + \inf_g \int_{\mathbb{X}} g^p \, d\nu \bigg)^{1/p},
\end{equation*}
where the infimum is taken over all upper gradients of $u$  (equivalently: over all $p$-weak upper gradients). Furthermore, if $\nu$ is doubling and a weak $(1,p)$-Poincar\'e inequality is satisfied, Lipschitz functions are dense in $W^{1,p}(\mathbb{X},d,\nu)$.  We will later apply the same definition to open subsets $\Omega \subset \mathbb{X}$.

For every $u \in W^{1,p}(\mathbb{X},d,\nu)$, there exists a minimal $p$-weak upper gradient $|Du| \in L^p(\mathbb{X},\nu)$, i.e. we have
\begin{equation*}
|Du| \leq g \quad \nu-\mbox{a.e.}
\end{equation*}
for all $p$-weak upper gradients $g \in L^p(\mathbb{X},\nu)$ (see \cite[Theorem 2.5]{BB}). It is unique up to a set of measure zero. In particular, we may simply plug in $|Du|$ in the infimum in the definition of the norm in $W^{1,p}(\mathbb{X},d,\nu)$. Moreover, in \cite{AGS1} (see also \cite{DiMarinoTh}) it was proved that on complete and separable metric spaces equipped with a doubling measure (or even with a bit weaker assumptions) the various definitions of Sobolev spaces are equivalent, but also that various definitions of $|Du|$ are equivalent, including the Cheeger gradient or the minimal $p$-relaxed slope of $u$.

 Also for functions of bounded variation there are there are several different ways to introduce them in metric measure spaces. However, again on complete and separable metric spaces equipped with a doubling measure again they are equivalent, see \cite{ADiM,DiMarinoTh}. In this paper, we follow the definition of total variation introduced by Miranda in \cite{Miranda1}. For $u \in L^1(\mathbb{X},\nu)$, we define the total variation of $u$ on an open set $\Omega \subset \mathbb{X}$ by the formula
\begin{equation}\label{dfn:totalvariationonmetricspaces}
\vert D u \vert_{\nu}(\Omega):= \inf \left\{ \liminf_{n \to \infty} \int_\Omega g_{u_n} \, d\nu \ : \ u_n \in Lip_{loc}( \Omega), \ u_n \to u \ \hbox{in} \ L^1(\Omega, \nu) \right\},
\end{equation}
where $g_{u_n}$ is a $1$-weak upper gradient of $u$.  If $\nu$ is doubling on $\Omega$ and $(\Omega,d,\nu)$ satisfies a weak $(1,1)$-Poincar\'e inequality,  by density of Lipschitz functions in $W^{1,1}(\Omega,d,\nu)$ in the definition above we may require that $u_n$ are Lipschitz instead of locally Lipschitz. Moreover, the total variation $|Du|_\nu(\mathbb{X})$ defined by formula \eqref{dfn:totalvariationonmetricspaces} is lower semicontinuous with respect to convergence in $L^1(\mathbb{X},\nu)$. The space of functions of bounded variation $BV(\mathbb{X},d,\nu)$ consists of all functions $u \in L^1(\mathbb{X},\nu)$ such that $|Du|_\nu(\mathbb{X}) < \infty$. It is a Banach space with respect to the norm
$$\Vert u \Vert_{BV(\mathbb{X},d,\nu)}:= \Vert u \Vert_{L^1(\mathbb{X},\nu)} + \vert D u \vert_{\nu}(\mathbb{X}).$$
Convergence in norm is often too much to ask when we deal with $BV$ functions, therefore we will frequently use the notion of strict convergence. We say that a sequence $\{ u_i \} \subset BV(\mathbb{X},d,\nu)$ {\it strictly converges} to $u \in BV(\mathbb{X},d,\nu)$, if $u_i \to  u$ in $L^1(\mathbb{X}, \nu)$ and $| Du_i |_\nu(\mathbb{X}) \to | Du |_\nu(\mathbb{X})$. Note that we may apply these definition also for open subsets of $\mathbb{X}$.

We turn our attention to the definition of the boundary measure of an open set in a metric measure space. A set $E \subset \mathbb{X}$ is said to be of finite perimeter if $\1_E \in BV(\mathbb{X},d, \nu)$, and its perimeter is defined as
$${\rm Per}_{\nu}(E):= \vert D \1_E \vert_{\nu}(\mathbb{X}).$$
If $U \subset \mathbb{X}$ is an open set, we define the perimeter of $E$ in $U$ as
$${\rm Per}_{\nu}(E, U):= \vert D \1_E \vert_{\nu}(U).$$
Another common way to define the boundary measure in metric measure spaces in the {\it codimension one Hausdorff measure}. Given a set $A \subset \mathbb{X}$, it is defined as
$$\mathcal{H}(A):= \lim_{R \to 0} \inf \left\{ \sum_{i=1}^\infty \frac{\nu(B(x_i,r_i))}{r_i} \ : \ A \subset \bigcup_{i=1}^\infty B(x_i,r_i), \ 0 < r_i \leq R \right\}.$$
It is known from \cite[Theorem 5.3]{Ambrosio} that if $E \subset \mathbb{X}$ is of finite perimeter, then for any Borel set $A \subset \mathbb{X}$,
\begin{equation}\label{Ammb}
\frac{1}{C} \mathcal{H} (A \cap \partial_* E) \leq {\rm Per}_{\nu}(E, A) \leq C \mathcal{H} (A \cap \partial_* E),
\end{equation}
where $\partial_* E$ is the measure theoretical boundary of $E$, that is, the collection of all points $x \in \mathbb{X}$ for which simultaneously
$$\limsup_{r \to 0^+} \frac{\nu(B(x,r) \cap E)}{\nu(B(x,r))} >0, \quad  \limsup_{r \to 0^+} \frac{\nu(B(x,r) \setminus E)}{\nu(B(x,r))} >0.$$
In particular, if $\partial_*\Omega = \partial\Omega$, the spaces $L^p(\partial\Omega,\mathcal{H})$, $L^p(\partial\Omega, |D\1_{\Omega}|)$ coincide as sets for every $p \in [1,\infty]$, and are equipped with equivalent norms. We will write explicitly which norm we use every time where it is not clear from the context.

Definition of boundary values of BV functions in a metric measure space is a more delicate issue. We will restrict our attention to open sets and adopt the following definition, used for instance in \cite{LS}:
	
\begin{definition}\label{dfn:trace}
{\rm Let $\Omega \subset \mathbb{X}$ be an open set and let $u$ be a $\nu$-measurable function on $\Omega$. A number $T_\Omega u(x)$ is a {\it trace} of $u$ at $x \in \partial \Omega$ if
$$\lim_{r \to 0^+} \dashint_{\Omega \cap B(x,r)} \vert u - T_\Omega u(x) \vert \, d \nu = 0.$$
We say that $u$ has a trace in $\partial \Omega$ if $ T_\Omega u(x)$ exists for $\mathcal{H}$-almost every $x \in \partial \Omega$.}
\end{definition}

Well-posedness of the trace and identifying the trace space of $W^{1,1}(\Omega,d,\nu)$ or $BV(\Omega,d,\nu)$ in the setting of metric measure spaces is not immediate and requires additional structural assumptions on $\Omega$. We summarize the results known in the literature in the following Theorem, which is a combination of \cite[Theorem 1.2]{MSS} and \cite[Theorem 5.5]{LS}.

\begin{theorem}\label{thm:traces}
Suppose that $\nu$ is doubling and $\mathbb{X}$ supports a weak $(1,1)$-Poincar\'e inequality. Let $\Omega$ be an open bounded set which supports a weak $(1,1)$-Poincar\'e inequality. Assume that $\Omega$ additionally satisfies the {\it measure density condition}, i.e.
there is a constant $C > 0$ such that
\begin{equation}\label{MDC1}
\nu(B(x,r) \cap \Omega) \geq C \nu(B(x,r))
\end{equation}
for $\mathcal{H}$-a.e. $x \in \partial \Omega$ and every $r \in (0, {\rm diam}(\Omega))$. Moreover, assume that {\it $\partial\Omega$ is Ahlfors codimension 1 regular}, i.e. there is a constant $C > 0$ such that
\begin{equation}\label{eq:codimensiononeregularity}
C^{-1} \frac{\nu(B(x,r))}{r} \leq \mathcal{H}(B(x,r) \cap \partial\Omega) \leq C \frac{\nu(B(x,r))}{r}
\end{equation}
for all $x \in \partial\Omega$ and every $r \in (0,\mbox{diam}(\Omega))$. \\
\\
Under these assumptions, Definition \ref{dfn:trace} defines an operator $T_\Omega: BV(\Omega,d,\nu) \twoheadrightarrow L^1(\partial\Omega,\mathcal{H})$. Moreover, the operator $T_\Omega$ is linear, bounded and surjective.
\end{theorem}

In particular, when $\mathbb{X}$ is complete, these assumptions imply that $\nu(\partial\Omega) = 0$ and $\mathcal{H}(\partial\Omega) < \infty$. Furthermore, under the same assumptions there is a (nonlinear) bounded extension operator $\mbox{Ext}: L^1(\partial\Omega,\mathcal{H}) \rightarrow BV(\Omega,d,\nu)$ such that $T_\Omega \circ \mbox{Ext}$ is the identity operator on $L^1(\partial\Omega,\mathcal{H})$. As the discussion in \cite{MSS} shows, any of these conditions cannot be dropped if we want the trace operator to be surjective; however, in order for it to be a linear and bounded operator with values in $L^1(\partial\Omega,\mathcal{H})$ we may weaken the assumptions a little bit and only assume that the upper bound in \eqref{eq:codimensiononeregularity} holds. Let us also note that when $\partial\Omega$ is Ahlfors codimension one regular, we have $\partial_* \Omega = \partial\Omega$, so the spaces $L^p(\partial\Omega,\mathcal{H})$ and $L^p(\partial\Omega,|D\chi_\Omega|_\nu)$ coincide as sets and have equivalent norms.

In the paper, we will often rely on approximation of arbitrary BV functions by more regular functions, requiring that the trace $T_\Omega$ of the approximating sequence is the same as the trace of the desired function. The following Lemma was proved in \cite[Lemma 3.2]{GM2021-2}; it is an upgraded version of \cite[Corollary 6.7]{LS}.

\begin{lemma}\label{lem:goodapproximation}
Suppose that $\Omega$ satisfies the assumptions of Theorem \ref{thm:traces}  and let $u \in BV(\Omega,d, \nu)$. Assume that $\mathcal{H}(\partial\Omega) < \infty$. Then, there exist locally Lipschitz functions $u_n \in \mbox{Lip}_{loc}(\Omega) \cap BV(\Omega,d,\nu)$ such that: \\
$(1)$ $u_n \rightarrow u$ strictly in $BV(\Omega,d,\nu)$ and $T_\Omega u_n = T_\Omega u$ $\mathcal{H}$-a.e.;
\\
$(2)$ Let $p \in [1,\infty)$. If $u \in L^p(\Omega,\nu)$, then $u_n \in L^p(\Omega,\nu)$ and $u_n \rightarrow u$ in $L^p(\Omega,\nu)$;
\\
$(3)$ If $u \in L^\infty(\Omega,\nu)$, then $u_n \in L^\infty(\Omega,\nu)$ and $u_n \rightharpoonup u$ weakly* in $L^\infty(\Omega,\nu)$.
\end{lemma}

\subsection{The differential structure}\label{subsec:diffstructure}

We need to work in a setting in which there is a Gauss-Green formula which is valid not only for Lipschitz functions, but for functions of bounded variation; this is the key element of the characterisation of solutions to the total variation flow in Sections \ref{Neumann} and \ref{sec:dirichlet}. We follow Gigli (see \cite{Gig}) and Buffa-Comi-Miranda (see \cite{BCM}) in the introduction of first-order differential structure on a metric measure space $(\mathbb{X},d,\nu)$. From now on, we assume that $\mathbb{X}$ is a complete and separable metric space and $\nu$ is a nonnegative Radon measure. We will introduce additional assumptions on $(\mathbb{X},d,\nu)$ in due course; usually, we will require some set to satisfy the assumptions of Theorem \ref{thm:traces}.

\begin{definition}{\rm
The cotangent module to $\mathbb{X}$ is defined as
$$ \mbox{PCM}_p = \left\{ \{(f_i, A_i)\}_{i \in \N} \ : \ (A_i)_{i \in\N} \subset \mathcal{B}(\mathbb{X}), \  f_i \in D^{1,p}(A_i), \ \ \sum_{i \in \N} \int_{A_i} |Df_i|^p \, d\nu < \infty  \right\},$$
where $A_i$ is a partition of $\mathbb{X}$. We define the equivalence relation $\sim$ as
$$\{(A_i, f_i)\}_{i \in \N} \sim \{(B_j,g_j)\}_{j \in \N} \quad \mbox{if} \quad |D(f_i - g_j)| = 0 \ \ \nu-\hbox{a.e. on} \ A_i \cap B_j.$$
The map $\vert \cdot \vert_* : \mbox{PCM}_p/\sim \rightarrow L^p(\mathbb{X}, \nu)$, well-defined $\nu$-a.e. on $A_i$ for all $i \in \N$ and given by
$$\vert \{(f_i, A_i)\}_{i \in \N} \vert_* := \vert D f_i \vert$$
is called the {\it pointwise norm} on $ \mbox{PCM}_p/\sim$.

In $ \mbox{PCM}_p/\sim$ we define the norm $\| \cdot \|$ as
$$ \|  \{(f_i, A_i)\}_{i \in \N} \|^p = \sum_{i\in \N} \int_{A_i}|Df_i|^p $$
and set $L^p(T^* \mathbb{X})$ to be the closure of $\mbox{PCM}_p / \sim$ with respect to this norm, i.e. we identify functions which differ by a constant and we identify possible rearranging of the sets $A_i$. $L^p(T^* \mathbb{X})$ is called the {\it cotangent module} and its elements will be called {\it $p$-cotangent vector field}.

$L^p(T^* \mathbb{X})$ is a $L^p(\nu)$-normed module, we denote by  $L^q(T\mathbb{X})$ the dual module of $L^p(T^* \mathbb{X})$, namely $L^q(T\mathbb{X}):= \hbox{HOM}(L^p(T^* \mathbb{X}), L^1(\mathbb{X}, \nu))$, which is a $L^q(\nu)$-normed module. The elements of $L^q(T\mathbb{X})$ will be called  {\it $q$-vector fields} on $\mathbb{X}$. The duality between $\omega \in L^p(T^* \mathbb{X})$ and $L \in  L^q(T\mathbb{X})$ will be denoted by $\omega(X) \in L^1(\mathbb{X}, \nu)$. Since the module $L^p(T^* \mathbb{X})$ is reflexive we can identify
$$ L^q(T\mathbb{X})^* = L^p(T^*\mathbb{X}),$$
where $\frac{1}{p} + \frac{1}{q} = 1$.}

\end{definition}
	
\begin{definition}\label{dfn:differential}
{\rm
Given $f \in D^{1,p}(\mathbb{X})$, we can define its {\it differential} $df$ as an element of $L^p(T^* \mathbb{X})$ given by the formula $df = (f, \mathbb{X})$.}

\end{definition}

Clearly, the operation of taking the differential is linear as an operator from $D^{1,p}(\mathbb{X})$ to $L^p(T^{*} \mathbb{X})$; moreover, from the definition of the norm in $L^p(T^{*} \mathbb{X})$ it is clear that this operator is bounded with norm equal to one.

For a vector field in $L^q(T\mathbb{X})$, it is also possible to define its divergence via an integration by parts formula, in in the case when it can be represented by an $L^p$ function or a Radon measure. Let us recall the definition given in \cite{BCM,DiMarinoTh}. We set
$$ \mathcal{D}^q(\mathbb{X}) = \left\{ X \in L^q(T\mathbb{X}): \,\, \exists f \in L^q(\mathbb{X}) \,  \,\, \int_\mathbb{X} fg d\nu = - \int_{\mathbb{X}} dg(X) d\nu \ \ \forall g \in  W^{1,p}(\mathbb{X},d,\nu) \right\}. $$
Here, the right hand side makes sense as an action of an element of $L^p(T^* \mathbb{X})$ on an element of $L^{q}(T\mathbb{X})$; the resulting function is an element of $L^1(\mathbb{X},\nu)$. The function $f$, which is unique by the density of  $W^{1,p}(\mathbb{X},d,\nu)$ in $ L^{p}(\mathbb{X},\nu)$, will be called the {\it $q$-divergence} of the vector field $X$, and we shall write $\mbox{div}(X)= f$.
	
It is easy to see that given $X \in \mathcal{D}^q(\mathbb{X})$ and $f \in D^{1,p}(\mathbb{X}) \cap L^\infty(\mathbb{X},\nu)$ with $\vert D f \vert \in L^\infty(\mathbb{X},\nu)$, we have
\begin{equation}\label{eq:leibnizformula}
fX \in \mathcal{D}^q(\mathbb{X}) \quad \hbox{and} \quad \mbox{div}(fX) = df(X) + f \mbox{div}(X).
\end{equation}
 We will also need to consider vector fields whose divergence is integrable with a different exponent than the vector field itself. For $\frac1r + \frac1s = 1$, we set
$$ \mathcal{D}^{q,r}(\mathbb{X}) = \bigg\{ X \in L^q(T\mathbb{X}): \,\, \exists f \in L^r(\mathbb{X},\nu) \quad \forall g \in W^{1,p}(\mathbb{X},d,\nu) \cap L^s(\mathbb{X},\nu) \qquad $$
$$ \qquad\qquad\qquad\qquad\qquad\qquad\qquad\qquad\qquad\qquad\qquad \int_{\mathbb{X}} fg \, d\nu = - \int_{\mathbb{X}} dg(X) \, d\nu \bigg\}. $$
We still write $\mbox{div}(X) = f$. Typically, we will require $q = \infty$ and $r = 1,2$.

In the course of the paper, we will extensively rely on the first order differential structure presented above. It is well-defined on metric spaces which are complete and separable; moreover, the structure is not defined locally (at least a priori; for some positive results in this direction see for instance \cite{LP}) - the objects $T^* \mathbb{X}$ and $T \mathbb{X}$ are not well-defined (the notation $L^p(T^* \mathbb{X})$ and $L^q(T\mathbb{X})$ is purely formal) and there is no immediate way to localise it to an open set $\Omega \subset \mathbb{X}$. However, whenever $\Omega \subset \mathbb{X}$ is an open bounded set, then $\overline{\Omega}$ is also a complete and separable metric space; hence, the whole first-order differential structure described above may be defined on $\overline{\Omega}$ as well. However, under the assumptions of Theorem \ref{thm:traces}, we may identify Newton-Sobolev functions on $\Omega$ and $\overline{\Omega}$ (and also BV functions, see \cite[Remark 2.8]{GM2021-2}). Then, on $\overline{\Omega}$ the Newton-Sobolev space is equivalent to the Sobolev space defined by test-plans as in \cite{Gig} and \cite{BCM}, see \cite[Theorem B.4]{Gig2}. Using this identification, we may also define the differential structure on $\Omega$ if it is sufficiently regular; with a slight abuse of notation, we write  $L^p(T^* \Omega)$ and $L^q(T\Omega)$, even though technically these objects are defined via an isometric extension to $\overline{\Omega}$.

However, under this identification, the definition of the divergence introduced above is not suitable for our purposes. The reason is that it takes into account the boundary effects. In order to have a relaxed version of the Gauss-Green formula, we need to use a notion of divergence which will only see the structure of $X$ inside the open set $\Omega$ (even though it is not clearly defined locally). The solution we suggest is to test the definition of the divergence using only functions which vanish at the boundary. Given an open bounded set $\Omega \subset \mathbb{X}$ which satisfies the assumptions of Theorem \ref{thm:traces}, we set
$$ \mathcal{D}_0^q(\Omega) = \left\{ X \in L^q(T\Omega): \, \exists f \in L^q(\Omega,\nu) \,\, \int_\Omega fg d\nu = - \int_{\Omega} dg(X) d\nu \ \ \forall g \in  W^{1,p}_0(\Omega,d,\nu) \right\},$$
where $W^{1,p}_0(\Omega,d,\nu)$ is the space of Sobolev functions in $W^{1,p}(\Omega,d,\nu)$ with zero trace. We again say that the (uniquely defined) function $f$ is the divergence of $X$ (when it is clear from the context) and we write $\mbox{div}_0(X) = f$. The relationship between the two definitions of the divergence can be roughly described as follows: by Theorem \cite[Theorem 3.6]{GM2021-2}, the divergence $\mbox{div}(X)$ is the divergence $\mbox{div}_0(X)$ plus a boundary term which has an interpretation of the normal trace.

Similarly, for $\frac1r + \frac1s = 1$, we set
$$ \mathcal{D}_0^{q,r}(\Omega) = \bigg\{ X \in L^q(T\Omega): \,\, \exists f \in L^r(\Omega,\nu) \quad \forall g \in  W_0^{1,p}(\Omega,d,\nu) \cap L^s(\Omega,\nu) \qquad $$
$$ \qquad\qquad\qquad\qquad\qquad\qquad\qquad\qquad\qquad\qquad\qquad \int_\Omega fg \, d\nu = - \int_{\Omega} dg(X) \, d\nu \bigg\}. $$
We still write $\mbox{div}_0(X) = f$. The divergence $\mbox{div}_0$ also has property \eqref{eq:leibnizformula}, i.e. whenever $X \in \mathcal{D}_0^q(\Omega)$ and $f \in D^{1,p}(\Omega) \cap L^\infty(\Omega,\nu)$ with $|Df| \in L^\infty(\Omega,\nu)$, we have
\begin{equation*}
fX \in \mathcal{D}_0^q(\Omega) \quad \hbox{and} \quad \mbox{div}_0(fX) = df(X) + f \mbox{div}_0(X).
\end{equation*}
Also, whenever $X \in L^\infty(T\Omega)$ with $\mbox{div}_0(X) \in L^q(\Omega,\nu)$, we have
\begin{equation}\label{eq:leibnizformulav2}
fX \in L^\infty(T\Omega), \quad \mbox{div}_0(fX) \in L^q(\Omega,\nu) \quad \hbox{and} \quad \mbox{div}_0(fX) = df(X) + f \mbox{div}_0(X).
\end{equation}

\subsection{Anzellotti pairings and Gauss-Green formula}

In order to provide a generalised Anzellotti pairing which satisfies a Green's formula, we will first need the Green's formula for Lipschitz functions in metric measure spaces. Recently, it was proved in \cite{BCM} under the assumption that $\Omega$ satisfies the following regularity assumption. Here, denote $$\Omega_t = \{ x \in \Omega: \, \mathrm{dist}(x, \Omega^c) \geq t \}.$$

\begin{definition}\label{def:regulardomain}
An open set $\Omega \subset \mathbb{X}$ is a {\it regular domain} if it has finite perimeter and
$$ |D\1_\Omega|(\mathbb{X}) = \limsup_{t \rightarrow 0} \frac{\nu(\Omega \setminus \Omega_t)}{t}. $$
\end{definition}

The reason why we will require our domain to be regular in the course of the paper is the validity of a Gauss-Green theorem.  Such a result for Lipschitz functions was first proved in \cite[Theorem 4.13]{BCM} (see also \cite[Theorem 3.6]{GM2021-2}); hovewer, for our purposes we need a Gauss-Green formula which is valid also for BV functions. To this end, we recall the notion of Anzellotti pairings introduced in \cite{GM2021,GM2021-2}, which are the metric space analogue of the classic Anzellotti pairings introduced in \cite{Anz}. The definitions on the whole space $\mathbb{X}$ and a sufficiently regular open bounded subset $\Omega \subset \mathbb{X}$ differ a bit; we will recall the definition on open bounded sets.

Suppose that $(\mathbb{X},d)$ is complete and separable, and $\Omega \subset \mathbb{X}$ satisfies the assumptions of Theorem \ref{thm:traces} (actually, a bit less is required, for details we refer to \cite{GM2021-2}). In particular, by \cite[Proposition 3.1]{BB} the space $\mathbb{X}$ is locally compact. Assume that $X \in L^\infty(T \Omega)$ and $u \in BV(\Omega, d, \nu)$. As in the case of classical Anzellotti pairings, we will additionally assume  a joint regularity condition on $u$ and $X$ which makes the pairing well-defined. The condition is as follows: for $p \in [1,\infty)$, we have
\begin{equation}\label{Anzellotti:assumption}
\mbox{div}_0(X) \in L^p(\Omega,\nu), \quad u \in BV(\Omega,d, \nu) \cap L^{q}(\Omega,\nu), \quad \frac{1}{p} + \frac{1}{q} = 1.
\end{equation}
 In other words, we assume that $X \in \mathcal{D}_0^{\infty,p}(\Omega)$.

\begin{definition}\label{dfn:Anzellotti}
Suppose that the pair $(X, u)$ satisfies the condition \eqref{Anzellotti:assumption}. Then, given a Lipschitz function  $f \in \mbox{Lip}(\Omega)$ with compact support, we set
$$ \langle (X, Du), f \rangle := -\int_\Omega u \, \mbox{div}_0(fX) \, d\nu =  -\int_\Omega u \, df(X) \, d\nu - \int_\Omega u f  \mbox{div}_0(X) \, d\nu.$$
\end{definition}
	
 The following bound on the pairing $(X,Du)$ in terms of the total variation of $u$ was proved in \cite[Proposition 5.3]{GM2021} (see also \cite{GM2021-2}).

\begin{proposition}\label{prop:boundonAnzellottipairing}
Suppose that $\Omega$ satisfies the assumptions of Theorem \ref{thm:traces}. Then, $(X, Du)$ is a Radon measure which is absolutely continuous with respect to $|Du|_\nu$.  Moreover, for every Borel set $A \subset \Omega$ we have
$$ \int_A |(X,Du)| \leq \| X \|_\infty \int_A |Du|_\nu.$$
\end{proposition}

 The following result proved in \cite[Theorem 3.11]{GM2021-2} motivates the construction of the Anzellotti pairings presented above. Namely, it states that the Gauss-Green formula can be extended to the setting of BV functions in place of Lipschitz functions.
	
\begin{theorem}\label{thm:generalgreensformula}
 Let $\Omega \subset \mathbb{X}$ be a bounded regular domain which  satisfies the assumptions of Theorem \ref{thm:traces}. Suppose that the pair $(X,u)$ satisfies  the condition \eqref{Anzellotti:assumption}. Then
\begin{equation}\label{Green1}
\int_\Omega u \,   \mbox{div}_0(X) \, d\nu  + \int_\Omega (X,Du) =  - \int_{\partial\Omega}  T_\Omega u \, (X \cdot \nu_\Omega)^- \, d| D_{\1_{\Omega}} |_\nu.
\end{equation}
\end{theorem}

By Proposition \ref{prop:boundonAnzellottipairing}, the pairing $(X,Du)$ is absolutely continuous with respect to $|Du|_\nu$. We will denote by $\theta(X,Du, \cdot)$ the Radon-Nikodym derivative of $(X,Du)$ with respect to $|Du|_\nu$, i.e. $\theta(X,Du, \cdot): \Omega \rightarrow \R$ is a $|Du|_\nu$-measurable function such that
\begin{equation}\label{RN1}\int_B (X,Du) = \int_B \theta(X,Du, x) \, d |Du|_\nu \quad \hbox{for all Borel sets } \ B \subset \Omega.
\end{equation}
Moreover,
$$\Vert \theta(X,Du, \cdot) \Vert_{L^\infty(\Omega, |Du|_\nu)} \leq \Vert X \Vert_\infty.$$
Furthermore, working as in the proof of \cite[Corollary 5.7]{GM2021}, except with $\mbox{div}_0$ in place of $\mbox{div}$, we get the following result.

\begin{proposition}\label{Lipcomp}
Let $\Omega \subset \mathbb{X}$ be a bounded regular domain which  satisfies the assumptions of Theorem \ref{thm:traces}. If $T: \R \rightarrow \R$ is a Lipschitz continuous increasing function, then
$$\theta(X,D(T \circ u), x) = \theta(X,Du, x)\quad \vert Du \vert_\nu-\hbox{a.e. in} \ \ \Omega.$$

\end{proposition}

\section{The Neumann Problem for the Total Variation Flow }\label{Neumann}

In this section we study  the Neumann problem
\begin{equation}\label{Neumann100}
\left\{ \begin{array}{lll} u_t (t,x) = {\rm div} \left(\frac{ D u (t,x)}{\vert Du(t,x) \vert_\nu } \right)  \quad &\hbox{in} \ \ (0, T) \times \Omega, \\[5pt] \frac{\partial u}{\partial \eta}:= \frac{Du}{\vert Du \vert_\nu}\cdot \eta = 0 \quad &\hbox{in} \ \ (0, T) \times  \partial\Omega, \\[5pt] u(0,x) = u_0(x) \quad & \hbox{in} \ \  \Omega. \end{array} \right.
\end{equation}

This Section is organised into two parts. In the first one, we study the gradient flow of the total variation with Neumann boundary data and prove existence, uniqueness and characterisation of weak solutions for $L^2$ initial data. Then, we introduce the notion of entropy solutions to deal with $L^1$ initial data.

\subsection{Weak solutions}

 Consider the energy functional $\mathcal{TV}_N : L^2(\Omega, \nu) \rightarrow [0, + \infty]$ defined by
\begin{equation}\label{NequmannF1}
\mathcal{TV_N}(u):= \left\{ \begin{array}{ll} \vert Du \vert_\nu (\Omega)  \quad &\hbox{if} \ u \in BV(\Omega, d, \nu) \cap L^2(\Omega,\nu), \\ \\ + \infty \quad &\hbox{if} \ u \in  L^2(\Omega, \nu) \setminus BV(\Omega, d, \nu).\end{array}\right.
\end{equation}
It is clear that $\mathcal{TV_N}$ is convex and lower semi-continuous with respect to the $L^2(\Omega,\nu)$-convergence. Then, by the theory of maximal monotone operators (see \cite{Brezis}) there is a unique strong solution of the abstract Cauchy problem
\begin{equation}\label{eq:abstractcauchyproblemneumann}
\left\{ \begin{array}{ll} u'(t) + \partial \mathcal{TV_N}(u(t)) \ni 0, \quad t \in [0,T] \\[5pt] u(0) = u_0. \end{array}\right.
\end{equation}

To characterize the subdifferential of $\mathcal{TV_N}$, we define the following operator.

\begin{definition}{\rm $(u,v) \in \mathcal{A_N}$ if and only if $u, v \in L^2(\Omega, \nu)$, $u \in BV(\Omega, d, \nu)$ and  there exists a vector field $X \in  \mathcal{D}_0^{\infty,2}(\Omega)$ with $\| X \|_\infty \leq 1$ such that the following conditions hold:
$$  -\mbox{div}_0(X) = v \quad \hbox{in} \ \Omega; $$
$$ (X, Du) = |Du|_\nu \quad \hbox{as measures};$$
$$ (X \cdot \nu_\Omega)^- =0 \qquad  | D_{\1_{\Omega}} |_\nu-a.e. \ \hbox{on} \ \partial \Omega.$$
}
\end{definition}

 Our main goal in this Section is to prove that $\mathcal{A}_{\mathcal{N}}$ coincides with the subdifferential of $\mathcal{TV}_{\mathcal{N}}$ and study some consequences of this result. To get this characterisation, we need to use the version of the Fenchel-Rockafellar duality Theorem given in \cite[Remark 4.2]{EkelandTemam}.

Let $U,V$ be two Banach spaces and let $A: U \rightarrow V$ be a continuous linear operator. Denote by $A^*: V^* \rightarrow U^*$ its dual. Then, if the primal problem is of the form
\begin{equation}\tag{P}\label{eq:primal}
\inf_{u \in U} \bigg\{ E(Au) + G(u) \bigg\},
\end{equation}
then the dual problem is defined as the maximisation problem
\begin{equation}\tag{P*}\label{eq:dual}
\sup_{p^* \in V^*} \bigg\{ - E^*(-p^*) - G^*(A^* p^*) \bigg\},
\end{equation}
where $E^*$ and $G^*$ are the Legendre–Fenchel transformations (conjugate  functions) of $E$ and $G$ respectively, i.e.,
$$E^* (u^*):= \sup_{u \in U} \left\{ \langle u, u^* \rangle - E(u) \right\}.$$

\begin{theorem}[Fenchel-Rockafellar Duality Theorem]\label{FRTh} Assume that $E$ and $G$ are proper, convex and lower semi-continuous. If there exists $u_0 \in U$ such that $E(A u_0) < \infty$, $G(u_0) < \infty$ and $E$ is continuous at $A u_0$, then
$$\inf \eqref{eq:primal} = \sup \eqref{eq:dual}$$
and the dual problem \eqref{eq:dual} admits at least one solution. Moreover, the optimality condition of these two problems is given by
\begin{equation}\label{optimality}
A^* p^* \in \partial G(\overline{u}), \quad -p^* \in \partial E(A\overline{u})),
\end{equation}
where $\overline{u}$ is solution of \eqref{eq:primal} and $p^*$ is solution of \eqref{eq:dual}.
\end{theorem}

 In the case when there is no solution to the dual problem, instead of optimality conditions we have the $\varepsilon-$subdifferentiability property of minimising sequences, see \cite[Proposition V.1.2]{EkelandTemam}: for any minimising sequence $u_n$ for \eqref{eq:primal} and a maximiser $p^*$ of \eqref{eq:dual}, we have
\begin{equation}\label{eq:epsilonsubdiff11N}
0 \leq E(Au_n) + E^*(-p^*) - \langle -p^*, Au_n \rangle \leq \varepsilon_n
\end{equation}
\begin{equation}\label{eq:epsilonsubdiff21N}
0 \leq G(u_n) + G^*(A^* p^*) - \langle u_n, A^* p^* \rangle \leq \varepsilon_n
\end{equation}
with $\varepsilon_n \rightarrow 0$.

 It will be possible to prove the characterisation of solutions using the Fenchel-Rockafellar duality theorem, because the Gigli differential structure is linear and our assumptions on the domain guarantee that the trace operator is linear and bounded. The main reason to rely on duality theory is that the differential structure is (at least a priori) not defined locally and some of the Euclidean tools fail; for instance, we cannot approximate a vector field with integrable divergence by more regular vector fields, because it is not clear how the vector fields change from point to point.

\begin{theorem}\label{thm:neumann}
Suppose that $\Omega$ and $\mathbb{X} \backslash \overline{\Omega}$ satisfy the assumptions of Theorem \ref{thm:traces}. Then,
$$\partial \mathcal{TV_N} = \mathcal{A_N},$$
and $D(\mathcal{A_N})$ is dense in $L^2(\Omega,\nu)$.
\end{theorem}
\begin{proof} The operator $\mathcal{A_N}$ is monotone. In fact, given $(u_i, v_i) \in \mathcal{A_N}$, $i = 1, 2$, there exist vector fields $X_i \in  \mathcal{D}_0^{\infty,2}(\Omega)$ with $\| X_i \|_\infty \leq 1$ such that the following conditions hold:
$$  -\mbox{div}_0(X_i) = v_i \quad \hbox{in} \ \Omega; $$
$$ (X_i, Du_i) = |Du_i|_\nu \quad \hbox{as measures};$$
$$ (X_i \cdot \nu_\Omega)^- =0 \qquad  | D_{\1_{\Omega}} |_\nu-a.e. \ \hbox{on} \ \partial \Omega.$$
Then, applying  the Gauss-Green formula (Theorem \ref{thm:generalgreensformula}), we have
$$\int_\Omega (v_1 - v_2) (u_1 - u_2) \, d\nu = - \int_\Omega (u_1 - u_2) (\mbox{div}_0(X_1)- \mbox{div}_0(X_2)) \, d\nu $$
$$=\int_\Omega (X_1 - X_2, Du_1 - Du_2) =  \int_\Omega \vert Du_1 \vert_\nu - \int_\Omega (X_1, Du_2) +  \int_\Omega \vert Du_2 \vert_\nu - \int_\Omega (X_2, Du_1) \geq 0. $$

Let us see now that $\mathcal{A_N} \subset \partial \mathcal{TV_N}$. In fact, let $(u,v) \in \mathcal{A_N}$. Then, given $w \in BV(\Omega, d,\nu) \cap L^2(\Omega, \nu)$, we have
$$ \int_\Omega v(w-u) d \nu = - \int_\Omega   \mbox{div}_0(X)(w -u) \, d\nu = \int_\Omega (X, D(w-u))  = \int_\Omega (X, Dw)- \int_\Omega \vert Du \vert_\nu   $$ $$\leq \mathcal{TV_N}(w) - \mathcal{TV_N}(u),$$
and consequently, $(u,v) \in \partial \mathcal{TV_N}$.

Now, the operator $\partial \mathcal{TV_N}$ is maximal monotone. Then, if $\mathcal{A_N}$ satisfies the range condition, by Minty Theorem we would also have that the operator $\mathcal{A_N}$ is maximal monotone, and consequently $\partial \mathcal{TV_N}= \mathcal{A_N}$. In order to finish the proof, we need to show that $\mathcal{A_N}$ satisfies the range condition, i.e.
\begin{equation}\label{RCondN}
\hbox{Given} \ g \in  L^2(\Omega, \nu), \ \exists \, u \in D(\mathcal{A_N}) \ s.t. \ \  u + \mathcal{A_N}(u) \ni g.
\end{equation}
 We can rewrite the above as
$$u + \mathcal{A_N}(u) \ni g \iff (u, g-u) \in \mathcal{A_N},$$
so we need to show that  there exists a vector field $X \in \mathcal{D}_0^{\infty,2}(\Omega)$ with $\| X \|_\infty \leq 1$ such that the following conditions hold:
\begin{equation}\label{1RCondN}
 -\mbox{div}_0(X) = g-u \quad \hbox{in} \ \Omega;
\end{equation}
\begin{equation}\label{2RCondN}
(X, Du) = |Du|_\nu \quad \hbox{as measures};
\end{equation}
\begin{equation}\label{3RCondN}
(X \cdot \nu_\Omega)^- =0 \qquad  | D_{\1_{\Omega}} |_\nu-a.e. \ \hbox{on} \ \partial \Omega.
\end{equation}

We are going to prove \eqref{RCondN}  by means of the Fenchel-Rockafellar Duality Theorem.  We set $U = W^{1,1}(\Omega,d,\nu) \cap L^2(\Omega,\nu)$, $V = L^1(\partial\Omega,|D\1_{\Omega}|_\nu) \times L^1(T^{*} \Omega)$, and the operator $A: U \rightarrow V$ is defined by the formula
$$
Au = ( T_\Omega u, du),
$$
where  $T_\Omega: BV(\Omega,d,\nu) \rightarrow L^1(\partial\Omega, |D\chi_\Omega|_\nu)$ is the trace operator in the sense of Definition \ref{dfn:trace}  (by our assumptions, the target space coincides with $L^1(\partial\Omega,\mathcal{H})$ and has an equivalent norm) and $du$ is the differential of $u$ in the sense of Definition \ref{dfn:differential}. Hence, $A$ is a linear and continuous operator. Moreover, the dual spaces to $U$ and $V$ are
$$
U^* = (W^{1,1}(\Omega,d,\nu) \cap L^2(\Omega,\nu))^*, \qquad V^* =  L^\infty(\partial\Omega,|D\1_{\Omega}|_\nu) \times L^\infty(T\Omega).
$$
We denote the points $p \in V$ in the following way: $p = (p_0, \overline{p})$, where $p_0 \in L^1(\partial\Omega, |D\chi_\Omega|_\nu)$ and $\overline{p} \in L^1(T^{*} \Omega)$. We will also use a similar notation for points $p^* \in V^*$. Then, we set $E: L^1(\partial\Omega, |D\chi_\Omega|_\nu) \times L^1(T^{*} \Omega) \rightarrow \mathbb{R}$ by the formula
\begin{equation}\label{Ieq:definitionofEN}
E(p_0, \overline{p}) = E_0(p_0) + E_1(\overline{p}), \quad E_0(p_0) = 0, \quad E_1(\overline{p}) = \| \overline{p} \|_{L^1(T^{*}  \Omega)}.
\end{equation}
We also set $G:W^{1,1}(\Omega,d,\nu) \cap L^2(\Omega,\nu) \rightarrow \mathbb{R}$ by
$$G(u):= \frac12 \int_\Omega u^2 \, d\nu - \int_\Omega ug \, d\nu.$$
The functional $G^* : (W^{1,1}(\Omega,d,\nu) \cap L^2(\Omega,\nu))^* \rightarrow [0,+\infty ]$ is given by
$$G^*(u^*) = \displaystyle\frac12 \int_\Omega (u^*  + g)^2 \, d\nu.$$
Now, by the definition of the dual operator, for any $u \in W^{1,1}(\Omega,d,\nu) \cap L^2(\Omega,\nu)$ and any $p^* = (p^*_0,\overline{p}^*) \in L^\infty(\partial\Omega,|D\1_{\Omega}|_\nu) \times L^\infty(T\Omega)$  in the domain of $A^*$ we have
$$  \int_\Omega u \, (A^* p^*) \, d\nu =  \langle u, A^* p^* \rangle  = \langle p^*, Au \rangle = \int_{\partial\Omega} p^*_0 \, T_\Omega u \, d|D\1_{\Omega}|_\nu + \int_\Omega du(\overline{p}^*) \, d\nu. $$
 First, let us take $u \in W_0^{1,1}(\Omega,d,\nu) \cap L^2(\Omega,\nu)$. Then, the first integral equals zero, so
\begin{equation*}
\int_\Omega u \, (A^* p^*) \, d\nu = \int_\Omega du(\overline{p}^*) \, d\nu,
\end{equation*}
so the definition of the divergence is satisfied with
\begin{equation}\label{eq:neumanndiv0}
 \mbox{div}_0(\overline{p}^*) = - A^* p^*.
\end{equation}
In particular, $\overline{p}^* \in \mathcal{D}_0^{\infty,2}(\Omega)$. Hence, for any $u \in W^{1,1}(\Omega,d,\nu) \cap L^2(\Omega,\nu)$ we may apply the Green formula (Theorem \ref{thm:generalgreensformula}) and get
$$  \int_\Omega u \, (A^* p^*) \, d\nu = \langle u, A^* p^* \rangle  = \langle p^*, Au \rangle = \int_{\partial\Omega} p^*_0 \, T_\Omega u \, d|D\1_{\Omega}|_\nu + \int_\Omega du(\overline{p}^*) \, d\nu = $$
$$ = \int_{\partial\Omega} p^*_0 \, T_\Omega u \, d|D\1_{\Omega}|_\nu - \int_\Omega u \, \mbox{div}_0(\overline{p}^*) \, d\nu - \int_{\partial\Omega} T_\Omega u \, (\overline{p}^* \cdot \nu_\Omega)^- \, d| D_{\1_{\Omega}} |_\nu $$
$$ = - \int_\Omega u \, \mbox{div}_0(\overline{p}^*) \, d\nu + \int_{\partial\Omega} T_\Omega u \, (p_0^* - (\overline{p}^* \cdot \nu_\Omega)^-) \, d| D_{\1_{\Omega}} |_\nu.$$
By \eqref{eq:neumanndiv0}, the integrals over $\Omega$ cancel out,  so
$$ \int_{\partial\Omega} T_\Omega u \, (p_0^* - (\overline{p}^* \cdot \nu_\Omega)^-) \, d| D_{\1_{\Omega}} |_\nu = 0$$
for all $u \in W^{1,1}(\Omega,d,\nu) \cap L^2(\Omega,\nu)$. Since the trace operator $T_\Omega$ is onto as an operator from $W^{1,1}(\Omega,d,\nu)$ to $L^1(\partial\Omega,|D\chi_\Omega|_\nu)$, by considering truncations we see that after a restriction to $W^{1,1}(\Omega,d,\nu) \cap L^\infty(\Omega,\nu)$ its image is $L^\infty(\partial\Omega,|D\chi_\Omega|_\nu)$. Therefore, after a restriction to $W^{1,1}(\Omega,d,\nu) \cap L^2(\Omega,\nu)$ its image is dense in $L^1(\partial\Omega,|D\chi_\Omega|_\nu)$, because it contains $L^\infty(\partial\Omega,|D\chi_\Omega|_\nu)$. In particular, we have
$$ \int_{\partial\Omega} w \, (p_0^* - (\overline{p}^* \cdot \nu_\Omega)^-) \, d| D_{\1_{\Omega}} |_\nu = 0$$
for $w$ in a dense subset of $L^1(\partial\Omega,|D\chi_\Omega|_\nu)$. Hence,
\begin{equation}\label{eq:pandpoagree}
p_0^* = (\overline{p}^* \cdot \nu_\Omega)^- \qquad | D_{\1_{\Omega}} |_\nu-\mbox{a.e. on } \partial\Omega.
\end{equation}
Therefore,
\begin{equation}\label{eq:formulaforgstar}
G^*(A^*p^*) = \displaystyle\frac12 \int_\Omega ( -\mbox{div}_0(\overline{p}^*)  + g)^2 d\nu.
\end{equation}
It is clear that the functional $E_0^*: L^\infty(\partial\Omega,|D\chi_\Omega|_\nu) \rightarrow \mathbb{R} \cup \{ \infty \}$ is given by the formula
$$E_0^*(p_0^*) = \left\{ \begin{array}{ll} 0, \quad &\hbox{if} \ \ p_0^* = 0, \\[10pt] +\infty,\quad & p_0^* \not=  0  \end{array}  \right. $$
Let us see now that the functional $E_1^*:  L^\infty( T \Omega) \rightarrow [0,\infty]$ is given by the formula
$$E_1^*(\overline{p}^*) = \left\{ \begin{array}{ll}0, \quad &  \| \overline{p}^* \|_\infty \leq 1; \\[10pt] +\infty,\quad &\hbox{otherwise},\end{array}  \right. $$
i.e. $E_1^* = I_{B_1^\infty}$, the indicator function of the unit ball of  $L^\infty(T \Omega)$. In fact, we have $$(I_{B_1^\infty})^*(\overline{p}) =  \Vert \overline{p} \Vert_{L^1(T^* \Omega)},$$
and since $I_{B_1^\infty}$ is convex and lower semi-continuous, we have
$$I_{B_1^\infty} = (I_{B_1^\infty})^{**} = E_1^*. $$

 Now, we apply the Fenchel-Rockafellar duality theorem. With $E$ and $G$ defined as above, consider the primal problem of the form
\begin{equation}\label{primal1N}
 \inf_{u \in U} \bigg\{ E(Au) + G(u) \bigg\}.
\end{equation}
For $u_0 \equiv 0$ we have $E(Au_0) = 0 < \infty$, $G(u_0) = 0 < \infty$ and $E$ is continuous at $0$. Then, by the Fenchel-Rockafellar Duality Theorem, we have
\begin{equation}\label{DFR1-TVflowN}\inf \eqref{eq:primal} = \sup \eqref{eq:dual}
\end{equation}
and
\begin{equation}\label{DFR2-TVflowN} \hbox{the dual problem \eqref{eq:dual} admits at least one solution,}
\end{equation}
where the dual problem is
\begin{equation}\label{dual1N}
\sup_{p^* \in L^\infty(\partial\Omega, |D\chi_\Omega|_\nu) \times L^\infty(T  \Omega)} \bigg\{  - E_0^*(-p_0^*) - E_1^*(\overline{p}^*) - G^*(A^* p^*) \bigg\}.
\end{equation}
Keeping in mind the above calculations, we set $\mathcal{Z}$ to be the subset of $V^*$ such that the dual problem does not immediately return $-\infty$, namely
$$
\mathcal{Z} = \bigg\{ p^* \in L^\infty(\partial\Omega, |D\chi_\Omega|_\nu) \times \mathcal{D}_0^{\infty,2}(\Omega): \,  \, \| \overline{p}^* \|_\infty \leq 1;  \, p_0^* =0 \bigg\}.
$$

Hence, we may rewrite the dual problem as
\begin{equation}\label{dual3N}
\sup_{p^* \in \mathcal{Z}} \bigg\{ - G^*(A^* p^*)\bigg\}.
\end{equation}
In light of the constraint $p_0^* = (\overline{p}^* \cdot \nu_\Omega)^-$, we may simplify the dual problem a bit. We set
$$
\mathcal{Z}' = \bigg\{ \overline{p}^* \in  \mathcal{D}_0^{\infty,2}(\Omega): \, \, \| \overline{p}^* \|_\infty \leq 1, \ (\overline{p}^* \cdot \nu_\Omega)^-  =0\bigg\},
$$
and then, using formula \eqref{eq:formulaforgstar}, we may rewrite the dual problem as
\begin{equation}\label{dual4N}
 \sup_{\overline{p}^* \in \mathcal{Z}'} \bigg\{ - \displaystyle\frac12 \int_\Omega ( -\mbox{div}_0(\overline{p}^*)  + g)^2 d\nu \bigg\}.
\end{equation}

Now, consider the energy functional $\mathcal{G_N} : L^2(\Omega, \nu) \rightarrow (-\infty, + \infty]$ defined by
\begin{equation}\label{11fuctmeN}
\mathcal{G_N}(v):= \left\{ \begin{array}{ll} \mathcal{TV_N}(v) + G(v),   \quad &\hbox{if} \ v \in BV(\Omega, d, \nu) \cap L^2(\Omega, \nu), \\ \\ + \infty \quad &\hbox{if} \ v \in  L^2(\Omega, \nu) \setminus BV(\Omega, d, \nu).\end{array}\right.
\end{equation}
 This functional is the extension of the functional $E + G$, which is well-defined for functions in $W^{1,1}(\Omega,d,\nu) \cap L^2(\Omega,\nu)$, to the space $BV(\Omega,d,\nu) \cap L^2(\Omega,\nu)$. Since $\mathcal{G_N}$ is coercive, convex and lower semi-continuous, the minimisation problem
$$\min_{v \in L^2(\Omega, \nu)} \mathcal{G_N}(v)$$ admits an optimal solution $u$ and we have
\begin{equation}\label{minimun1N}
\min_{v \in L^2(\Omega, \nu)} \mathcal{G_N}(v) = \inf_{v \in U} \bigg\{ E(Av) + G(v) \bigg\}.
\end{equation}
Let us take a sequence $u_n \in W^{1,1}(\Omega,d,\nu)$ which has the same trace as $u$ and converges strictly to $u$ and also $u_n \to u$ in $L^2(\Omega, \nu)$ (given by Lemma \ref{lem:goodapproximation}); then, it is a minimising sequence in \eqref{eq:primal}.  Since by \eqref{DFR1-TVflowN} and \eqref{DFR2-TVflowN} there is no duality gap and there exists a solution to the dual problem, we can now apply to this sequence the $\varepsilon$-subdifferentiability property given in \eqref{eq:epsilonsubdiff11N} and \eqref{eq:epsilonsubdiff21N}. By equation \eqref{eq:epsilonsubdiff21N}, for every $w \in L^2(\Omega,\nu)$ we have
$$G(w) - G(u_n) \geq \langle (w -u_n), A^* p^* \rangle  -\varepsilon_n.$$
Hence,
$$G(w) - G(u) \geq \langle (w -u), A^* p^* \rangle,$$
and consequently,
$$ A^* p^* \in \partial G(u) = \{ u - g \}. $$
and by \eqref{eq:neumanndiv0}, we get
\begin{equation}\label{div1N}
 -\mbox{div}_0(\overline{p}^*) = u -g.
\end{equation}
 Also, by the definition of $\mathcal{Z}$ and keeping in mind that $p_0^* = (\overline{p}^* \cdot \nu_\Omega)^-$, we get
\begin{equation}\label{front1N}(-\overline{p}^* \cdot \nu_\Omega)^- =0\quad | D_{\1_{\Omega}} |_\nu-a.e. \ \hbox{on} \ \partial \Omega.
\end{equation}
On the other hand, equation \eqref{eq:epsilonsubdiff11N} gives
\begin{equation}\label{eq:goodN}
0 \leq \int_{\partial\Omega}  p_0^* T_\Omega u_n \, d|D\1_{\Omega}|_\nu + \bigg( \| du_n \|_{ L^1(T^{*} \Omega)} + \int_\Omega du_n(\overline{p}^*) \, d\nu \bigg) \leq \varepsilon_n.
\end{equation}
 Because the trace of $u_n$ is fixed (and equal to the trace of $u$), the integral on $\partial\Omega$ does not change with $n$; hence, it has to equal zero. Since $\| du_n \|_{L^1(T^{*} \Omega)} = \int_{\Omega} |Du_n|_\nu$, in the integral on $\Omega$ we have
\begin{equation}\label{inerior1N}
0 \leq \int_\Omega \bigg( |Du_n|_\nu + du_n(\overline{p}^*) \bigg) d\nu \leq \varepsilon_n.
\end{equation}
Finally, keeping in mind that  $-\mbox{div}_0(\overline{p}^*) = u-g$ and again using the fact that the trace of $u_n$ is fixed and equal to the trace of $u$, by Green's formula we get
$$ \int_\Omega du_n(\overline{p}^*) \, d\nu = - \int_\Omega u_n \,  \mbox{div}_0(\overline{p}^*) \, d\nu  =\int_\Omega u_n \, (u -g) \, d\nu . $$
Then, applying again Green's formula, we have
$$\lim_{n \to \infty}  \int_\Omega du_n(\overline{p}^*) \, d\nu = \int_\Omega u \, (u-g) \, d\nu  = -  \int_\Omega u \,  \mbox{div}_0(\overline{p}^*) \, d\nu  = \int_\Omega (\overline{p}^*, Du). $$

Hence, since $u_n$ converges strictly to $u$ and having in mind \eqref{inerior1N}, we get
$$ \int_\Omega |Du|_\nu - \int_\Omega (-p^*, Du) = \lim_{n \rightarrow \infty} \bigg( \int_\Omega |Du_n|_\nu - \int_\Omega (-\overline{p}^*, Du) \bigg) = 0.$$
This together with Proposition \ref{prop:boundonAnzellottipairing} implies that
\begin{equation}\label{measure1N}
(-\overline{p}^*, Du) = |Du|_\nu \quad \mbox{as measures in } \Omega.
\end{equation}
Then, as consequence of \eqref{div1N}, \eqref{measure1N} and \eqref{front1N}, we have that the pair $(u, -\overline{p}^*)$ satisfies \eqref{1RCondN}, \eqref{2RCondN} and \eqref{3RCondN} respectively, and therefore \eqref{RCondN} holds.

Finally, by \cite[Proposition 2.11]{Brezis}, we have
$$ D(\partial \mathcal{TV_N}) \subset  D(\mathcal{TV_N}) =  BV(\Omega,d,\nu) \cap L^2(\Omega,\nu) \subset \overline{D(\mathcal{TV_N})}^{L^2(\Omega, \nu)} \subset \overline{D(\partial \mathcal{TV_N})}^{L^2(\Omega, \nu)},$$
from which follows the density of the domain.
\end{proof}

We can also prove the following characterisation of weak solutions in terms of variational inequalities. In the next subsection, we will use a similar reasoning to characterise entropy solutions for $L^1$ initial data; in fact, a variant of condition $(c)$ will be the way we introduce the definition of entropy solutions.

\begin{corollary}\label{cor:characterisationweaksolutions}
The following conditions are equivalent: \\
$(a)$ $(u,v) \in \partial\mathcal{TV}_{\mathcal{N}}$; \\
$(b)$ $(u,v) \in \mathcal{A}_{\mathcal{N}}$, i.e. $u, v \in L^2(\Omega, \nu)$, $u \in BV(\Omega, d, \nu)$ and there exists a vector field  $X \in \mathcal{D}_0^{\infty,2}(\Omega)$ with  $\| X \|_\infty \leq 1$ such that $-\mbox{div}_0(X) = v$ in $\Omega$ and
\begin{equation}
(X,Du) = |Du|_\nu  \quad \hbox{as measures in } \Omega;
\end{equation}
\begin{equation}
(X \cdot \nu_\Omega)^- = 0 \qquad |D\chi_\Omega|_\nu-\mbox{a.e. on } \Omega;
\end{equation}
$(c)$ $u, v \in L^2(\Omega, \nu)$, $u \in BV(\Omega, d, \nu)$ and there exists a vector field $X \in \mathcal{D}_0^{\infty,2}(\Omega)$ with  $\| X \|_\infty \leq 1$ such that $-\mbox{div}_0(X) = v$ in $\Omega$ and for every $w \in BV(\Omega,d,\nu) \cap L^2(\Omega,\nu)$
\begin{equation}\label{eq:variationalinequalitytvflow}
\int_{\Omega} v(w-u) \, d\nu \leq \int_{\Omega} (X,Dw) - \int_{\Omega} |Du|_\nu;
\end{equation}
$(d)$ $u, v \in L^2(\Omega, \nu)$, $u \in BV(\Omega, d, \nu)$ and there exists a vector field  $X \in \mathcal{D}_0^{\infty,2}(\Omega)$ with  $\| X \|_\infty \leq 1$ such that $-\mbox{div}_0(X) = v$ in $\Omega$ and for every $w \in BV(\Omega,d,\nu) \cap L^2(\Omega,\nu)$
\begin{equation}
\int_{\Omega} v(w-u) \, d\nu = \int_{\Omega} (X,Dw) - \int_{\Omega} |Du|_\nu.
\end{equation}
\end{corollary}

\begin{proof}
The equivalence of $(a)$ and $(b)$ is given by Theorem \ref{thm:neumann}. To see that $(b)$ implies $(d)$, multiply the equation $v = -\mbox{div}_0(X)$ by $w - u$ and integrate over $\Omega$ with respect to $\nu$. Using the Gauss-Green formula, i.e. Theorem \ref{thm:generalgreensformula}, we get
\begin{equation*}
\int_{\Omega} v(w-u) \, d\nu = - \int_{\Omega} (w - u) \mbox{div}_0(X) \, d\nu = \int_{\Omega} (X,Dw) - \int_{\Omega} |Du|_\nu,
\end{equation*}
where the second equality follows from the fact that $(X \cdot \nu_\Omega) = 0$ $|D\chi_\Omega|_\nu$-a.e. on $\partial\Omega$, so the boundary term disappears.

Obviously, $(d)$ implies $(c)$. To finish the proof, let us see that $(c)$ implies $(b)$. If we take $w = u$ in \eqref{eq:variationalinequalitytvflow}, we get
\begin{equation*}
\int_{\Omega} |Du|_\nu \leq \int_{\Omega} (X,Du).
\end{equation*}
By Proposition \ref{prop:boundonAnzellottipairing}, this implies that $(X,Du) = |Du|_\nu$ as measures in $\Omega$. It remains to show that $(X \cdot \nu_\Omega)$ is the zero function. Notice that
\begin{equation}\label{eq:equalityofuvandDu}
\int_\Omega u v d\nu = \int_\Omega |Du|_\nu,
\end{equation}
where inequality in one direction is given by taking $w = 0$ in \eqref{eq:variationalinequalitytvflow} and in the other direction by taking $w = 2u$ in \eqref{eq:variationalinequalitytvflow}. But then, applying \eqref{eq:equalityofuvandDu} to condition \eqref{eq:variationalinequalitytvflow}, we get that for all $w \in BV(\Omega,d,\nu) \cap L^2(\Omega,\nu)$ we have
\begin{equation*}
\int_{\Omega} vw d\nu \leq \int_\Omega (X,Dw),
\end{equation*}
so actually we have an equality because the same inequality holds for $-w$. But then, using the Gauss-Green formula (Theorem \ref{thm:generalgreensformula}), for all $w \in BV(\Omega,d,\nu) \cap L^2(\Omega,\nu)$ we have
\begin{equation*}
\int_{\partial\Omega} T_\Omega w \, (X \cdot \nu_\Omega)^- \, d|D\chi_\Omega|_\nu = - \int_\Omega w \, \mbox{div}_0(X) \, d\nu - \int_\Omega (X,Dw) = \int_\Omega wv \, d\nu - \int_\Omega wv \, d\nu = 0.
\end{equation*}
Using a density argument as in the proof of equation \eqref{eq:pandpoagree}, we see that $(X \cdot \nu_\Omega)^- = 0$ $|D\chi_\Omega|_\nu$-a.e.
\end{proof}

\begin{definition}\label{def35N}
We define in $L^2(\Omega,\nu)$ the multivalued operator $\Delta^N_{1,\nu}$ by
\begin{center}$(u, v ) \in \Delta^N_{1,\nu}$ \ if, and only if, \ $-v \in \partial\mathcal{TV_N}(u)$.
\end{center}

\end{definition}

Our concept of solution of the Neumann problem \eqref{Neumann100} is the following:
\begin{definition}\label{def:neumann1p}
{\rm  Given $u_0 \in L^2(\Omega,\nu)$, we say that $u$ is a {\it weak solution} of the  Neumann problem \eqref{Neumann100} in $[0,T]$, if $u \in  W^{1,1}(0, T; L^2(\Omega,\nu))$,   $u(0, \cdot) =u_0$, and for almost all $t \in (0,T)$
\begin{equation}\label{def:neumannflow}
 u_t(t, \cdot) \in \Delta^N_{1,\nu}(t, \cdot).
\end{equation}
In other words, if $u(t) \in BV(\Omega, d, \nu)$ and there exist vector fields $X(t) \in   \mathcal{D}_0^{\infty,2}(\Omega)$ with $\| X(t) \|_\infty \leq 1$ such that, for almost all $t \in [0,T]$, the following conditions hold:
$$  \mbox{div}_0(X(t)) = u_t(t, \cdot) \quad \hbox{in} \ \Omega; $$
$$ (X(t), Du(t)) = |Du(t)|_\nu \quad \hbox{as measures};$$
$$ (X(t) \cdot \nu_\Omega)^- =0 \qquad  | D_{\1_{\Omega}} |_\nu-a.e. \ \hbox{on} \ \partial \Omega.$$
}
\end{definition}

Then, using the classical theory of maximal monotone operators (see for instance \cite{Brezis}), as a consequence of Theorem \ref{thm:neumann} we have the following existence and uniqueness theorem.

\begin{theorem}
For any $u_0 \in L^2(\Omega,\nu)$ and all $T > 0$, there exists a unique  weak solution of the Neumann problem \eqref{Neumann100} in $[0,T]$.
\end{theorem}

Finally, let us comment on the asymptotic behaviour of weak solutions to the Neumann problem. We will rely on the results proved by Bungert and Burger in \cite{BuBu} for the gradient flow of a coercive (in the sense described below), $1$-homogeneous convex functional defined on a Hilbert space. These results are a recent generalisation of the classical results obtained in \cite{ACDM} (see also \cite{ACMBook}) for the total variation flow in Euclidean spaces.

Let us recall some notation used in \cite{BuBu}. Let $\mathcal{H}$ be a Hilbert space and $J: \mathcal{H} \rightarrow (-\infty, + \infty]$ a proper, convex, lower semi-continuous functional. Then, it is well known (see \cite{Brezis}) that the abstract Cauchy problem
\begin{equation}\label{ACP123}
\left\{ \begin{array}{ll} u'(t) + \partial J(u(t)) \ni 0, \quad t \in [0,T] \\[5pt] u(0) = u_0, \end{array}\right.
\end{equation}
has a unique strong solution $u(t)$ for any initial datum $u_0 \in \overline{D(J)}$.

We say that $J$ is {\it coercive}, if there exists a constant $C >0$ such that
\begin{equation}\label{coerc}
\Vert u \Vert \leq C J(u), \quad \forall \, u\in \mathcal{H}_0,
\end{equation}
where
$$\mathcal{H}_0:= \{ u \in \mathcal{H} \ : \ J(u)=0 \}^{\perp} \setminus \{ 0 \}.$$
Clearly, this inequality is equivalent to the following positive lower bound
$$\lambda_1(J):= \inf_{u \in \mathcal{H}_0} \frac{J(u)}{\Vert u \Vert} > 0.$$
The number $\lambda_1(J)$ defined above is called the {\it Rayleigh quotient} associated with $J$.

We will now comment on how to apply this framework to the total variation flow with Neumann boundary conditions. Clearly, the functional $\mathcal{TV}_{\mathcal{N}}$ is convex, $1$-homogenous and lower semi-continuous. Notice that
$$L^2(\Omega, \nu)_0:= \{ u \in L^2(\Omega, \nu): \ \mathcal{TV_N}(u)=0 \}^{\perp} \setminus \{ 0 \} = \left\{ u \in L^2(\Omega, \nu), \ u \not=0, \int_\Omega u d\nu =0\ \right\}.$$
Hence, $\mathcal{TV_N}$ is coercive if and only if if there exists a constant $C >0$ such that
\begin{equation}\label{coercN}
\Vert u \Vert_{L^2(\Omega, \nu)} \leq C \mathcal{TV_N}(u), \quad \forall \, u\in L^2(\Omega, \nu)_0,
\end{equation}
which is equivalent to the following {\it Poincar\'{e} inequality}
\begin{equation}\label{Poincare1}
\Vert u - \overline{u} \Vert_{L^2(\Omega, \nu)} \leq C \vert Du \vert_\nu (\Omega) , \quad \forall \, u\in BV(\Omega, d, \nu) \cap L^2(\Omega,\nu),
\end{equation}
where
$$\overline{u}:= \frac{1}{\nu(\Omega)} \int_\Omega u d\nu.$$
If the above Poincar\'e inequality holds, we can apply \cite[Theorem 2.2]{BuBu} and \cite[Theorem 2.3]{BuBu} to get the following result.

\begin{theorem}\label{mainResultNeu}
Assume that \eqref{Poincare1} holds. For  $u_0 \in L^2(\Omega,\nu)$, let  $u(t)$ be the weak  solution  the Neumann problem \eqref{Neumann100}. Then, we have
\begin{itemize}
\item[(i)] (Finite extinction time)
$$T_{\rm ex}(u_0) \leq \frac{\Vert u_0 \Vert_{L^2(\Omega, \nu)}}{\lambda_1(\mathcal{TV_N})},$$
where
$$T_{\rm ex}(u_0):= \inf \{ T >0 \ : \ u(t) = \overline{u_0},  \ \ \forall \, t \geq T \}.$$
\item[(ii)] (Asymptotic profiles for finite extinction) Let $$w(t):= \frac{u(t)}{\left(1 - \frac{1}{T_{\rm ex}(u_0)} t \right)},$$
and assume that $w(t)$ converges (possibly up to a subsequence) strongly to some $w_* \in L^2(\Omega, \nu)$ as $t \to T_{\rm ex}(u_0)$. Then it holds $$ \frac{1}{T_{\rm ex}(u_0)} \in \partial \mathcal{TV_N}(w_*),\quad   w_*  \not=0, \quad  \Vert w_* \Vert_{L^2(\Omega, \nu)} \leq \Vert u_0 \Vert_{L^2(\Omega, \nu)}.$$
\item[(iii)] (Ground state as asymptotic profile) An asymptotic profile $w_*$ is a ground state, i.e., $ w_* = {\rm argmin} \frac{\mathcal{TV_N}(w_*)}{\Vert w_* \Vert}$, if and only if $\lim_{t \nearrow T_{\rm ex}}(u_0) = \lambda_1(\mathcal{TV_N}).$
\end{itemize}

\end{theorem}

\subsection{Entropy solutions}

Now we are going to study the Neumann problem for initial data in $L^1(\Omega,\nu)$. For that, we will extensively use the following truncation function $T_k: \R \rightarrow \R$, defined for all $k > 0$ as
$$T_k(r):= \left\{ \begin{array}{lll} k, \quad &\hbox{if} \ r> k, \\[10pt] r, \quad &\hbox{if} \ \vert r \vert \leq k, \\[10pt] - k, \quad &\hbox{if} \ r < -k. \end{array} \right.$$
\begin{definition}
{\rm $(u,v) \in \mathcal{B_N}$ if and only if $u, v \in L^1(\Omega, \nu)$, $T_k(u) \in BV(\Omega, d, \nu)$ for all $k>0$,  and  there exists a vector field $X \in \mathcal{D}_0^{\infty,1}(\Omega)$ with $\| X \|_\infty \leq 1$ and $-\mbox{div}_0(X) = v$ in $\Omega$ such that
\begin{equation}\label{e1def}
\int_\Omega (w - T_k(u)) \,v \, d \nu \leq \int_\Omega (X, Dw) - \int_\Omega \vert DT_k(u) \vert_\nu
\end{equation}
for all $w \in  BV(\Omega, d, \nu) \cap L^\infty(\Omega, \nu)$ and $k > 0$.
}
\end{definition}

We have the following characterisation of the operator $\mathcal{B_N}$,  analogous to the characterisation of $\mathcal{A_N}$ given in Corollary \ref{cor:characterisationweaksolutions}. In fact, we will later see that $\mathcal{B_N}$ coincides with $\mathcal{A_N}$ after a restriction to its domain.

\begin{proposition}\label{charact1} The following conditions are equivalent: \\
$(a)$ $(u,v) \in \mathcal{B_N}$; \\
$(b)$ $u, v \in L^1(\Omega, \nu)$, $T_k(u) \in BV(\Omega, d, \nu)$ for all $k>0$,  and  there exists a vector field $X \in \mathcal{D}_0^{\infty,1}(\Omega)$ with $\| X \|_\infty \leq 1$ and  $ -\mbox{div}_0(X) = v$ in $\Omega$ such that
\begin{equation}\label{e2def}
\int_\Omega (w - T_k(u)) \, v \, d \nu = \int_\Omega (X, Dw) - \int_\Omega \vert DT_k(u) \vert_\nu
\end{equation}
for all $w \in  BV(\Omega, d, \nu) \cap L^\infty(\Omega, \nu)$ and $k >0$;\\
$(c)$  $u, v \in L^1(\Omega, \nu)$, $T_k(u) \in BV(\Omega, d, \nu)$ for all $k>0$,  and  there exists a vector field $X \in \mathcal{D}_0^{\infty,1}(\Omega)$ with $\| X \|_\infty \leq 1$ such that
\begin{equation}-\mbox{div}_0(X) = v \mbox{ in } \Omega;
\end{equation}
\begin{equation}\label{e3def}
\int_\Omega (X, DT_k(u)) = \int_\Omega \vert DT_k(u) \vert_\nu, \quad \forall \, k>0,
\end{equation}
and
\begin{equation}\label{e4def}
(X \cdot \nu_\Omega)^- =0 \qquad  | D_{\1_{\Omega}} |_\nu-a.e. \ \hbox{on} \ \partial \Omega.
\end{equation}
\end{proposition}

\begin{proof}
Obviously (b) implies (a). Let us see that (c) implies (b): multiply the equation $v = -\mbox{div}_0(X)$ by $w - T_k(u)$ and integrate over $\Omega$ with respect to $\nu$. Using the Gauss-Green formula, i.e. Theorem \ref{thm:generalgreensformula}, we get
\begin{equation*}
\int_{\Omega} v \, (w-T_k(u)) \, d\nu = - \int_{\Omega} (w - T_k(u)) \, \mbox{div}_0(X) \, d\nu = \int_{\Omega} (X,Dw) - \int_{\Omega} |DT_k(u)|_\nu,
\end{equation*}
where the second equality follows from the fact that $|D\chi_\Omega|_\nu$-a.e. on $\partial\Omega$ we have $(X \cdot \nu_\Omega)^- = 0$, so the boundary term disappears.

Let us see that (a) implies (c). Taking $w = T_k(u)$ in \eqref{e1def}, we get
$$0 \leq \int_\Omega (X, DT_k(u)) - \int_\Omega \vert DT_k(u) \vert_\nu.$$
Thus,
$$\int_\Omega (X, DT_k(u)) \leq  \Vert X \Vert_\infty \int_\Omega \vert DT_k(u) \vert_\nu \leq  \int_\Omega \vert DT_k(u) \vert_\nu \leq \int_\Omega (X, DT_k(u)),$$
and \eqref{e3def} follows.

On the other hand, by taking $w = 0$ in \eqref{e1def} we get
$$ \int_\Omega \vert DT_k(u) \vert_\nu \leq \int_\Omega T_k(u) v d \nu,$$
and by taking $w = 2T_k(u)$ in \eqref{e1def} and using \eqref{e3def}, we obtain
$$\int_\Omega T_k(u) v d \nu \leq 2\int_\Omega (X, DT_k(u)) - \int_\Omega\vert DT_k(u) \vert_\nu = \int_\Omega\vert DT_k(u) \vert_\nu.$$
Consequently,
\begin{equation}\label{e5def}
\int_\Omega T_k(u) v d \nu = \int_\Omega\vert DT_k(u) \vert_\nu \quad \hbox{for all} \ k>0.
\end{equation}
Now, using \eqref{e5def} in \eqref{e1def}, we have
$$\int_\Omega w v d \nu \leq  \int_\Omega (X, Dw), \quad \forall \, w \in  BV(\Omega, d, \nu) \cap L^\infty(\Omega, \nu).$$
Since the same inequality holds for $-  w \in  BV(\Omega, d, \nu) \cap L^\infty(\Omega, \nu)$, we obtain that
\begin{equation}\label{e6def}
\int_\Omega w v d \nu  = \int_\Omega (X, Dw), \quad \forall \, w \in  BV(\Omega, d, \nu) \cap L^\infty(\Omega, \nu).
\end{equation}
Finally, since $ -\mbox{div}_0(X) = v$ in $\Omega$, by Gauss-Green formula (Theorem \ref{thm:generalgreensformula}) and \eqref{e6def} we have
$$ \int_{\partial\Omega}  T_\Omega w \, (X \cdot \nu_\Omega)^- \, d| D_{\1_{\Omega}} |_\nu = - \int_\Omega w \,   \mbox{div}_0(X) \, d\nu  - \int_\Omega (X,Dw) = 0,$$
for all  $w \in  BV(\Omega, d, \nu) \cap L^\infty(\Omega, \nu)$.  Using a density argument as in the proof of equation \eqref{eq:pandpoagree}, we see that \eqref{e4def} holds.
\end{proof}

\begin{corollary}\label{inclusion1} We have
$$\mathcal{B_N} \cap (L^2(\Omega,\nu) \times L^2(\Omega,\nu)) = \mathcal{A_N}.$$
\end{corollary}

\begin{proof}
Given $(u,v) \in \mathcal{B_N} \cap (L^2(\Omega,\nu) \times L^2(\Omega,\nu))$, we have $u, v \in L^2(\Omega,\nu)$, $T_k(u) \in BV(\Omega, d, \nu)$ for all $k>0$,  and  there exists a vector field $X \in \mathcal{D}_0^{\infty,1}(\Omega)$ with $\| X \|_\infty \leq 1$ and  $ -\mbox{div}_0(X) = v$ in $\Omega$ such that
\begin{equation}\label{e1defN}
\int_\Omega (w - T_k(u)) \, v \, d \nu \leq \int_\Omega (X, Dw) - \int_\Omega \vert DT_k(u) \vert_\nu
\end{equation}
for all $w \in  BV(\Omega, d, \nu) \cap L^\infty(\Omega, \nu)$ and $k >0$. Then, since  $T_k(u) \to u$ in $L^2(\Omega,\nu)$, by the lower semi-continuity of the total variation we have $u \in BV(\Omega, d, \nu)$, and letting $k \to \infty$ in \eqref{e1defN} we get
$$\int_\Omega (w -u)v d \nu \leq \int_\Omega (X, Dw) - \int_\Omega \vert Du \vert_\nu$$
for all $w \in  BV(\Omega, d, \nu) \cap L^\infty(\Omega, \nu)$ and $k >0$.  Then, by implication $(c) \Rightarrow (b)$ of Corollary \ref{cor:characterisationweaksolutions}, we obtain that $(u,v) \in  \mathcal{A_N}$: by considering truncations we may only test inequality \eqref{eq:variationalinequalitytvflow} with functions from $BV(\Omega,d,\nu) \cap L^\infty(\Omega,\nu)$, which is dense in $BV(\Omega,d,\nu) \cap L^2(\Omega,\nu)$.

Suppose now that $(u, v) \in \mathcal{A_N}$. Then, $u, v \in L^2(\Omega, \nu)$, $u \in BV(\Omega, d, \nu)$ and  there exists a vector field $X \in  \mathcal{D}_0^{\infty,2}(\Omega)$ with $\| X \|_\infty \leq 1$ such that the following conditions hold:
$$  -\mbox{div}_0(X) = v \quad \hbox{in} \ \Omega; $$
$$ (X, Du) = |Du|_\nu \quad \hbox{as measures};$$
$$ (X \cdot \nu_\Omega)^- =0 \qquad  | D_{\1_{\Omega}} |_\nu-a.e. \ \hbox{on} \ \partial \Omega.$$
Hence, by Proposition \ref{charact1}, to see that $(u,v) \in \mathcal{B_N}$, we only need to show that
\begin{equation}\label{e3defN}
\int_\Omega (X, DT_k(u)) = \int_\Omega \vert DT_k(u) \vert_\nu \quad \forall \, k>0,
\end{equation}
By Proposition \ref{Lipcomp}, we have
$$\theta(X, DT_k(u), x) = \theta(X, Du, x), \quad a.e. \ \hbox{with respect to} \ \vert DT_k(u) \vert_\nu \ \hbox{and} \ \vert Du \vert_\nu.$$
Then, since $(X, Du) = |Du|_\nu \quad \hbox{as measures}$, we have $\theta(X, Du, x) = 1$ a.e. with respect to the measure $\vert Du \vert_\nu$. Hence
$$\int_\Omega (X, DT_k(u)) = \int_\Omega  \theta(X, DT_k(u), x) \vert DT_k(u) \vert_\nu = \int_\Omega   \vert DT_k(u) \vert_\nu.$$
\end{proof}

\begin{theorem}\label{L1Neumann}
The operator $\mathcal{B_N}$ is $m$-completely accretive in $L^1(\Omega, \nu)$ and homogeneous of degree zero. Moreover, $D(\mathcal{B_N})$ is dense in $L^1(\Omega, \nu)$.
\end{theorem}

\begin{proof}
First, note that  $\mathcal{B_N}$ is $L^1(\Omega, \nu)$ and homogeneous of degree zero, that is, for all $u \in D(\mathcal{B_N})$ and $\lambda \geq 0$, $(u, v) \in \mathcal{B_N}$ implies that $(\lambda u, v) \in \mathcal{B_N}$. This is consequence of the fact that
$$T_k(r) = \lambda T_{\frac{k}{\lambda}}(r), \quad \forall \, \lambda \geq 0, \ \ r \in \R.$$

Let $P_{0}$ denote the set of all functions $p\in C^{\infty}(\R)$ satisfying $0\le p'\le 1$ such that $p'$  is compactly supported, and $x=0$ is not contained in the support $\textrm{supp}(p)$ of $p$. To prove that $\mathcal{B_N}$ is  a completely accretive operator we must show that  (see \cite{ACMBook,BCr2})
\begin{equation}\label{e1CA} \int_{\Omega}p(u_1-u_2) \, (v_1-v_2)\, d\nu\ge 0 \end{equation}
for every $p\in P_{0}$ and every $(u_i,v_i) \in \mathcal{B_N}$, $i=1,2$. In fact, given $(u_i, v_i) \in \mathcal{B_N}$, $i =1,2$, we have that there exists a vector fields $X_i \in \mathcal{D}_0^{\infty,1}(\Omega)$ with $\| X_i \|_\infty \leq 1$ such that the following conditions hold:
$$ -\mbox{div}_0(X_i) = v_i \quad \hbox{in} \ \Omega; $$
$$ (X_i, DT_k(u_i)) = |DT_k(u_i)|_\nu \quad \hbox{as measures};$$
$$ (X_i \cdot \nu_\Omega)^- =0 \qquad  | D_{\1_{\Omega}} |_\nu-a.e. \ \hbox{on} \ \partial \Omega.$$
Making a similar computation as for the monotonicity of  $\mathcal{A}_N$ in the proof of Theorem \ref{thm:neumann}, we get that for every Borel set $B \subset \Omega$ we have
$$\int_{B} (X_1 - X_2, D T_k(u_1) - DT_k(u_2)) $$ $$=  \int_{B} \vert DT_k(u_1) \vert_\nu - \int_{B} (X_1, DT_k(u_2)) +  \int_{B} \vert DT_k(u_2) \vert_\nu  - \int_{B} (X_2, DT_k(u_1)) \geq 0. $$
Hence, by \eqref{RN1},
$$ \int_{B} \theta(X_1 - X_2, D (T_k(u_1) - T_k(u_2)), x ) \, d \vert D(T_k(u_1) - T_k(u_2)) \vert_\nu $$ $$= \int_{B} (X_1 - X_2, D (T_k(u_1) - T_k(u_2))) \geq 0$$
for all Borel sets $B \subset \Omega$. Thus
$$\theta(X_1 - X_2, D (T_k(u_1) - T_k(u_2)), x ) \geq 0, \quad \vert D(T_k(u_1) - T_k(u_2)) \vert_\nu-\hbox{a.e. on} \ \Omega.$$
Now by Proposition \ref{Lipcomp}, we have that
$$\theta(X_1 - X_2, D (T_k(u_1) - T_k(u_2)), x )  = \theta(X_1 - X_2, D p(T_k(u_1) - T_k(u_2)), x ) $$
a.e. with respect to the measures $\vert D(T_k(u_1) - T_k(u_2)) \vert_\nu$ and $\vert Dp(T_k(u_1) - T_k(u_2)) \vert_\nu$. We conclude that
$$\theta(X_1 - X_2, D p(T_k(u_1) - T_k(u_2)), x ) \geq 0, \quad \vert Dp(T_k(u_1) - T_k(u_2)) \vert_\nu-\hbox{a.e. on} \ \Omega.$$
Now, we apply the Gauss-Green Formula (Theorem \ref{thm:generalgreensformula}). By Proposition \ref{Lipcomp} and the fact that $|D\chi_\Omega|_\nu$-a.e. we have $(X_i \cdot \nu_\Omega)^- = 0$, so the boundary terms disappear, we get
$$ \int_{\Omega}p(T_k(u_1)-T_k(u_2))(v_1-v_2)\, d\nu = - \int_\Omega p(T_k(u_1)-T_k(u_2)) (\mbox{div}(X_1)- \mbox{div}(X_2)) \, d\nu $$ $$=\int_\Omega (X_1 - X_2, D p(T_k(u_1)-T_k(u_2)) $$ $$= \int_\Omega \theta(X_1 - X_2, Dp(T_k(u_1)-T_k(u_2)) \, d \vert Dp(T_k(u_1)-T_k(u_2)) \vert_\nu \geq 0.$$
The inequality \eqref{e1CA} follows by letting $k \to \infty$. Therefore $\mathcal{B_N}$ is completely accretive.

We claim now that $\mathcal{B_N}$ is a closed operator. Let $(u_n, v_n) \in \mathcal{B_N}$ be such that $u_n \to u$ and $v_n \to v$ in $L^1(\Omega, \nu)$. Then, $T_k(u_n) \in BV(\Omega, d, \nu)$ for all $k>0$, and  there exist  vector fields $X_n \in \mathcal{D}_0^{\infty,1}(\Omega)$ with $\| X_n \|_\infty \leq 1$ such that the following conditions hold:
$$ -\mbox{div}_0(X_n) = v_n \quad \hbox{in} \ \Omega; $$
\begin{equation}\label{e1closed}\int_\Omega (w - T_k(u_n)) \, v_n \, d \nu \leq \int_\Omega (X_n, Dw) - \int_\Omega \vert DT_k(u_n)
\vert_\nu,\end{equation}
for all $w \in  BV(\Omega, d, \nu) \cap L^\infty(\Omega, \nu)$ and $k >0$.

By extracting a subsequence, if necessary, we can assume that
$$X_n \rightharpoonup X \quad \hbox{weakly$^*$ in} \ L^\infty(T\Omega).$$
Given $g \in W_0^{1,1}(\Omega, d, \nu)$, since $dg \in L^1(T^*\Omega)$, we have
$$\int_\Omega dg(X) d\nu = \lim_{n \to \infty} \int_\Omega dg(X_n) d\nu = - \lim_{n \to \infty} \int_\Omega  {\rm div}(X_n) g d \nu $$ $$= -\lim_{n \to \infty} \int_\Omega g v_n d \nu = - \int_\Omega g v d\nu.$$
Thus,
\begin{equation}\label{e1L1Neumann}- {\rm div}_0(X) = v,\quad \hbox{in} \ \ \Omega.
\end{equation}
Fix any $k > 0$. Taking $w =0$ in \eqref{e1closed}, we have that
$$\int_\Omega \vert DT_k(u_n) \vert_\nu \leq \int_\Omega T_k(u_n)v_n d \nu \leq k \cdot \sup \Vert v_n \Vert_{L^1(\Omega, \nu)}.$$
 Since $u_n \rightarrow u$ in $L^1(\Omega,\nu)$, the right-hand side is uniformly bounded, so the left-hand side is also uniformly bounded. By the lower semi-continuity of the total variation, we have that $T_k(u) \in BV(\Omega, d, \nu)$ and
$$\int_\Omega  |DT_k(u)|_\nu \leq \liminf_{n \to \infty} \int_\Omega |DT_k(u_n)|.$$
Then, letting $n \to \infty$ in \eqref{e1closed} we obtain that
$$\int_\Omega (w - T_k(u)) \, v \, d \nu \leq \int_\Omega (X, Dw) - \int_\Omega \vert DT_k(u) \vert_\nu.$$
Hence $(u,v) \in \mathcal{B_N}$, so $\mathcal{B_N}$ is closed.

Let us see now that the range condition
\begin{equation}\label{rangecond1}
R(I + \mathcal{B_N}) = L^1(\Omega, \nu)
\end{equation}
holds. Since $\mathcal{B_N}$ is closed, it is enough to show that
\begin{equation}\label{rangecond2}
L^\infty(\Omega, \nu) \subset R(I + \mathcal{B_N}).
\end{equation}
Now, by the proof of Theorem \ref{thm:neumann}, we have
$$R(I + \mathcal{A_N}) = L^2(\Omega, \nu),$$
from where \eqref{rangecond2} holds, having in mind Corollary \ref{inclusion1}.

Finally, since $D(\mathcal{A_N})$ is dense in $L^2(\Omega,\nu)$, by Corollary \ref{inclusion1} we obtain that $D(\mathcal{B_N})$ is dense in $L^1(\Omega, \nu)$.
\end{proof}

 As a consequence of the proof, we also get that the operator $\mathcal{A_N}$ is completely accretive, since by Proposition \ref{inclusion1} it is a restriction of $\mathcal{B_N}$ to $L^2(\Omega, \nu) \times L^2(\Omega, \nu)$.

In \cite{HM} the authors studied the regularizing effect of the nonlinear semigroup generated by $m$-accretive operator of degree zero. More precisely, as consequence of \cite[Theorem 4.13]{HM}, the following result holds.

\begin{theorem}\label{regularity} If $A$ is an operator in $L^1(Z)$ is a $m$-completely accretive operator of degree zero and $(0,0) \in A$, then, for any $u_0 \in \overline{D(A)}$, the unique mild solution $u(t)$ of the abstract Cauchy problem
\begin{equation}\label{eq:abstractcauchyHM}
u'(t) \in  A(u(t)), \quad t \geq 0. \quad u(0) = u_0
\end{equation}
is a strong solution, that is, verifies \eqref{eq:abstractcauchyHM} for almost all $t \geq 0$, and moreover
$$\left\Vert \frac{du(t) }{dt} \right\Vert_{L^1(Z)} \leq \frac{\Vert u_0 \Vert_{L^1(Z)}}{t}\quad \hbox{for every} \ t >0,$$
and
$$ \frac{du(t) }{dt}  \leq \frac{u(t)}{t}\quad \hbox{a.e. on} \ Z  \ \hbox{for every} \ t >0, \ \hbox{if} \ u_0 \geq 0.$$
\end{theorem}

\begin{definition}{\rm
We say that $u \in C([0, T]; L^1(\Omega, \nu)) \cap W^{1,1}([0, T]; L^1(\Omega, \nu))$ is an {\it entropy solution} of the Neumann Problem \eqref{Neumann100} in $[0,T]$ with initial data $u_0 \in L^1(\Omega, \nu)$, if  for all $k > 0$ we have $T_k u(t) \in BV(\Omega, d, \nu)$ and there exist vector fields $X(t) \in  \mathcal{D}_0^{\infty,1}(\Omega)$ with $\| X(t) \|_\infty \leq 1$ such that, for almost all $t \in [0,T]$, the following conditions hold:
$$ \mbox{div}_0(X(t)) = u_t(t, \cdot) \quad \hbox{in} \ \Omega; $$
$$ (X(t), DT_k u(t)) = |DT_ku(t)|_\nu \quad \hbox{as measures};$$
$$ (X(t) \cdot \nu_\Omega)^- =0 \qquad  | D_{\1_{\Omega}} |_\nu-a.e. \ \hbox{on} \ \partial \Omega.$$
}
\end{definition}

Then, as consequence of Theorem \ref{L1Neumann} and Theorem \ref{regularity}, we have the following existence and uniqueness result.

\begin{theorem}\label{ExistUnique1}
For any $u_0 \in L^1(\Omega, \nu)$ and all $T > 0$ there is a unique entropy solution $u(t)$ of the Neumann problem \eqref{Neumann100} in $[0,T]$. Moreover, the following comparison principle holds: if $u_1, u_2$ are  entropy solutions for the initial data $u_{1,0}, u_{2,0} \in  L^q(\Omega, \nu)$, respectively, then
\begin{equation}\label{NCompPrincipleplaplaceBV}
\Vert (u_1(t) - u_2(t))^+ \Vert_q \leq \Vert ( u_{1,0}- u_{2,0})^+ \Vert_q \quad \hbox{for all} \ 1 \leq q \leq \infty.
\end{equation}
We also have
$$\left\Vert \frac{du(t) }{dt} \right\Vert_{L^1(\Omega,\nu)} \leq \frac{\Vert u_0 \Vert_{L^1(\Omega,\nu)}}{t}\quad \hbox{for every} \ t >0,$$
and  if $u_0 \geq 0$, then additionally
$$ \frac{du(t) }{dt}  \leq \frac{u(t)}{t}\quad  \nu-\hbox{a.e. on} \ \Omega  \ \hbox{for every} \ t > 0.$$
\end{theorem}

The comparison principle given in equation \eqref{NCompPrincipleplaplaceBV} is a consequence of the complete accretivity of the operator $\mathcal{B_N}$.

Actually, this notion of solutions for $L^1$ initial data can also be applied not only for the Neumann problem, but also for the total variation flow defined on the whole space, considered for instance in \cite{GM2021}. This is formalised in the following Remark.

\begin{remark}\label{rem:wholespace} {\rm
Suppose that $\nu(\mathbb{X}) < \infty$. Consider the energy functional
\begin{equation*}
\mathcal{TV}(u) = \twopartdef{|Du|_\nu}{u \in BV(\mathbb{X},d,\nu) \cap L^2(\mathbb{X},\nu)}{+\infty}{L^2(\mathbb{X},\nu) \setminus BV(\mathbb{X},d,\nu)}
\end{equation*}
and its gradient flow in $L^2(\mathbb{X},\nu)$
\begin{equation}\label{eq:onthewholespace}
\twopartdef{u'(t) + \partial \mathcal{TV}(u(t)) \ni 0,}{t \in [0,T];}{u(0) = u_0.}{}
\end{equation}
By theory of maximal monotone operators, there exists a unique solution of \eqref{eq:onthewholespace} for initial data $u_0 \in L^2(\mathbb{X},\nu)$ and its characterisation in terms of Anzellotti pairings was given in \cite{GM2021}. Working in a similar way as in this subsection, we may also introduce the notion of entropy solutions valid for initial data in $L^1(\mathbb{X},\nu)$; the proofs are actually simpler because the boundary terms do not appear in the computation. For the purpose of this Remark only, denote by $(X,Du)$ the Anzellotti pairing introduced in \cite{GM2021}, i.e. as in Definition \ref{dfn:Anzellotti}, but with the divergence $\mbox{div}$ in place of $\mbox{div}_0$.

With this understood, we say that $u \in C([0, T]; L^1(\mathbb{X}, \nu)) \cap W^{1,1}([0, T]; L^1(\mathbb{X}, \nu))$ is an {\it entropy solution} of the Cauchy problem \eqref{eq:onthewholespace} in $[0,T]$ with initial data $u_0 \in L^1(\mathbb{X}, \nu)$, if  for all $k > 0$ we have $T_k u(t) \in BV(\mathbb{X}, d, \nu)$ and there exist vector fields $X(t) \in  \mathcal{D}^{\infty,1}(\mathbb{X})$ with $\| X(t) \|_\infty \leq 1$ such that for almost all $t \in [0,T]$ the following conditions hold:
$$ \mbox{div}_0(X(t)) = u_t(t, \cdot) \quad \hbox{in} \ \mathbb{X}; $$
$$ (X(t), DT_k u(t)) = |DT_ku(t)|_\nu \quad \hbox{as measures}.$$
For any $u_0 \in L^1(\mathbb{X},\nu)$ and $T > 0$, there exists a unique entropy solution of the Cauchy problem \eqref{eq:onthewholespace} in $[0,T]$, and it satisfies the comparison principle and estimates given in Theorem \ref{ExistUnique1}.
}
\end{remark}

\section{The Dirichlet Problem for the Total Variation Flow}\label{sec:dirichlet}

In this section we study  the Dirichlet problem
\begin{equation}\label{Dirichlet100}
\left\{ \begin{array}{lll} u_t (t,x) = {\rm div} \left(\frac{ D u (t,x)}{\vert Du(t,x) \vert_\nu } \right)  \quad &\hbox{in} \ \ (0, T) \times \Omega, \\[5pt] u (t,x) = f(x) \quad &\hbox{in} \ \ (0, T) \times  \partial\Omega, \\[5pt] u(0,x) = u_0(x) \quad & \hbox{in} \ \  \Omega. \end{array} \right.
\end{equation}

Given $f \in L^1(\partial \Omega, \mathcal{H})$ we consider the energy functional $\mathcal{TV}_f : L^2(\Omega, \nu) \rightarrow [0, + \infty]$ defined by
\begin{equation}\label{L2fuctme}
\mathcal{TV}_f (u):= \left\{ \begin{array}{ll} \vert Du \vert_\nu (\Omega) + \displaystyle\int_{\partial \Omega} \vert  T_\Omega(u) - f \vert \, d |D\1_\Omega|_\nu   \quad &\hbox{if} \ u \in BV(\Omega, d, \nu) \cap L^2(\Omega,\nu), \\[10pt] + \infty \quad &\hbox{if} \ u \in  L^2(\Omega, \nu) \setminus BV(\Omega, d, \nu).\end{array}\right.
\end{equation}

 The following result was proved in \cite[Proposition 4.3]{GM2021-2}.

\begin{proposition}\label{semicont}
Suppose that $\Omega$ and $\mathbb{X} \backslash \overline{\Omega}$ satisfy the assumptions of Theorem \ref{thm:traces}. Then, the functional is convex and lower semicontinuous with respect to convergence in $L^1(\Omega,\nu)$.
\end{proposition}

By Proposition \ref{semicont},  if $\Omega$ and $\mathbb{X} \backslash \overline{\Omega}$ satisfy the assumptions of Theorem \ref{thm:traces}, we have that $\mathcal{TV}_f$ is lower semi-continuous with respect to the $L^2(\Omega,\nu)$-convergence. Then, by the theory of maximal monotone operators (see \cite{Brezis}) for every initial datum $u_0 \in \overline{D(\partial \mathcal{TV}_f)}^{L^2(\Omega,\nu)}$, there is a unique strong solution of the abstract Cauchy problem
\begin{equation}\label{eq:abstractcauchyproblemdirichlet}
\left\{ \begin{array}{ll} u'(t) + \partial \mathcal{TV}_f(u(t)) \ni 0, \quad t \in [0,T] \\[10pt] u(0) = u_0. \end{array}\right.
\end{equation}

To characterize the subdifferential of $\mathcal{TV}_f$, we define the following operator.

\begin{definition}{\rm $(u,v) \in \mathcal{A}_f$ if and only if $u, v \in L^2(\Omega, \nu)$, $u \in BV(\Omega, d, \nu)$ and  there exists a vector field  $X \in \mathcal{D}_0^{\infty,2}(\Omega)$ with $\| X \|_\infty \leq 1$ such that the following conditions hold:
$$ -\mbox{div}_0(X) = v \quad \hbox{in} \ \Omega; $$
$$ (X, Du) = |Du|_\nu \quad \hbox{as measures};$$
$$ (X \cdot \nu_\Omega)^- \in \mbox{sign}(T_\Omega u - f) \qquad  | D_{\1_{\Omega}} |_\nu-a.e. \ \hbox{on} \ \partial \Omega.$$

}
\end{definition}

To get the characterization of $\partial \mathcal{TV}_f$, we use again the Fenchel-Rockafellar duality Theorem (Theorem \ref{FRTh}). The outline of the proof will be similar to the one in Theorem \ref{thm:neumann}, with some additional difficulties arising from the appearance of boundary terms.

\begin{theorem}\label{thm:dirichlet}
Suppose that $\Omega$ and $\mathbb{X} \backslash \overline{\Omega}$ satisfy the assumptions of Theorem \ref{thm:traces}. Then,
$$\partial \mathcal{TV}_f = \mathcal{A}_f,$$
and $D(\mathcal{A}_f)$ is dense in $L^2(\Omega, \nu)$. Moreover, the operator $\mathcal{A}_f$ is $m$-completely accretive in $L^2(\Omega,\nu)$.
\end{theorem}

\begin{proof}
The operator $\mathcal{A}_f$ is monotone. In fact, given $(u_i, v_i) \in \mathcal{A}_f$, $i = 1, 2$, there exist vector fields $X_i \in \mathcal{D}_0^{\infty,2}(\Omega)$ with $\| X_i \|_\infty \leq 1$ such that the following conditions hold:
$$  -\mbox{div}_0(X_i) = v_i \quad \hbox{in} \ \Omega; $$
$$ (X_i, Du_i) = |Du_i|_\nu \quad \hbox{as measures};$$
$$ (X_i \cdot \nu_\Omega)^- \in \mbox{sign}(T_\Omega u_i-f) \qquad  | D_{\1_{\Omega}} |_\nu-a.e. \ \hbox{on} \ \partial \Omega.$$
Then, applying Green Formula, we have
$$\int_\Omega (v_1 - v_2) (u_1 - u_2) \, d\nu = - \int_\Omega (u_1 - u_2) (\mbox{div}_0(X_1)- \mbox{div}_0(X_2)) \, d\nu $$ $$=\int_\Omega (X_1 - X_2, Du_1 - Du_2) + \int_{\partial \Omega}  (T_\Omega u_1 -f + f - T_\Omega u_2) \, ((X_1 \cdot \nu_\Omega)^- - (X_2 \cdot \nu_\Omega)^-) \, d| D_{\1_{\Omega}} |_\nu $$
$$=  \int_\Omega \vert Du_1 \vert_\nu - \int_\Omega (X_1, Du_2) +  \int_\Omega \vert Du_2 \vert_\nu - \int_\Omega (X_2, Du_1) + \int_{\partial \Omega} \vert T_\Omega u_1 - f\vert \, d| D_{\1_{\Omega}} |_\nu $$
$$- \int_{\partial \Omega}(T_\Omega u_1 - f) (X_2 \cdot \nu_\Omega)^- \, d| D_{\1_{\Omega}} |_\nu  + \int_{\partial \Omega} \vert T_\Omega u_2 - f\vert \, d| D_{\1_{\Omega}} |_\nu  - \int_{\partial \Omega}(T_\Omega u_2 - f) (X_1 \cdot \nu_\Omega)^- \, d| D_{\1_{\Omega}} |_\nu, $$
and consequently,
$$\int_\Omega (v_1 - v_2) (u_1 - u_2) d\nu \geq 0.$$

Let us see now that $\mathcal{A}_f \subset \partial \mathcal{TV}_f$. In fact, let $(u,v) \in \mathcal{A}_f$. Then, given $w \in BV(\Omega, d,\nu) \cap L^2(\Omega, \nu)$, we have
$$ \int_\Omega v(w-u) d \nu = - \int_\Omega  \mbox{div}_0(X)(w -u) \, d\nu = \int_\Omega (X, D(w-u)) + \int_{\partial \Omega} (T_\Omega w - T_\Omega u) (X \cdot \nu_\Omega)^- \, d| D_{\1_{\Omega}} |_\nu $$
$$= \int_\Omega (X, Dw)- \int_\Omega \vert Du \vert_\nu  +  \int_{\partial \Omega} (T_\Omega w - f) (X \cdot \nu_\Omega)^- \, d| D_{\1_{\Omega}} |_\nu -   \int_{\partial \Omega} \vert T_\Omega u - f \vert \, d| D_{\1_{\Omega}} |_\nu  $$ $$\leq \mathcal{TV}_f(w) - \mathcal{TV}_f(u),$$
and consequently, $(u,v) \in \partial \mathcal{TV}_f$.

Now, by Proposition \ref{semicont}, $\partial \mathcal{TV}_f$ is maximal monotone. Then, if $\mathcal{A}_f$ satisfies the range condition, by Minty Theorem the operator $\mathcal{A}_f$ is maximal monotone and consequently $\partial \mathcal{TV}_f = \mathcal{A}_f$. In order to finish the proof, we need to see that $\mathcal{A}_f$ satisfies the range condition, i.e.
\begin{equation}\label{RCond}
\hbox{Given} \ g \in  L^2(\Omega, \nu), \ \exists \, u \in D(\mathcal{A}_f) \ s.t. \ \  u + \mathcal{A}_f(u) \ni g.
\end{equation}
 We can rewrite the above as
$$u + \mathcal{A}_f(u) \ni g \iff (u, g-u) \in \mathcal{A}_f,$$
so we need to show that there exists a vector field  $X \in \mathcal{D}_0^{\infty,2}(\Omega)$ with $\| X \|_\infty \leq 1$ such that the following conditions hold:
\begin{equation}\label{1RCond}
-\mbox{div}_0(X) = g-u \quad \hbox{in} \ \Omega;
\end{equation}
\begin{equation}\label{2RCond}
(X, Du) = |Du|_\nu \quad \hbox{as measures};
\end{equation}
\begin{equation}\label{3RCond}
(X \cdot \nu_\Omega)^- \in \mbox{sign}(T_\Omega u - f) \qquad  | D_{\1_{\Omega}} |_\nu-a.e. \ \hbox{on} \ \partial \Omega.
\end{equation}
As in the proof of Theorem \ref{thm:neumann}, we will prove \eqref{RCond} by means of the Fenchel-Rockafellar duality theorem. We will employ the same function spaces and operator $A$ as in that proof. We will also use the same functionals $E_1$ and $G$, the only difference being that the functional $E_0$ now incorporates the Dirichlet boundary condition. For convenience, we shortly present again the whole framework.

We set $U = W^{1,1}(\Omega,d,\nu) \cap L^2(\Omega,\nu)$, $V = L^1(\partial\Omega,|D\1_{\Omega}|_\nu) \times L^1(T^{*} \Omega)$, and the operator $A: U \rightarrow V$ is defined by the formula
\begin{equation*}
Au = ( T_\Omega u, du),
\end{equation*}
where  $T_\Omega: BV(\Omega,d,\nu) \rightarrow L^1(\partial\Omega, |D\chi_\Omega|_\nu)$ is the trace operator in the sense of Definition \ref{dfn:trace}(again, recall that under our assumptions the spaces $L^1(\partial\Omega,\mathcal{H})$ and $L^1(\partial\Omega, |D\chi_\Omega|_\nu)$ coincide as sets and have equivalent norms) and $du$ is the differential of $u$ in the sense of Definition \ref{dfn:differential}. Hence, $A$ is a linear and continuous operator. Moreover, the dual spaces to $U$ and $V$ are
\begin{equation*}
U^* = (W^{1,1}(\Omega,d,\nu) \cap L^2(\Omega,\nu))^*, \qquad V^* =  L^\infty(\partial\Omega,|D\1_{\Omega}|_\nu) \times L^\infty(T\Omega).
\end{equation*}
We denote the points $p \in V$ in the following way: $p = (p_0, \overline{p})$, where $p_0 \in L^1(\partial\Omega, |D\chi_\Omega|_\nu)$ and $\overline{p} \in L^1(T^{*} \Omega)$. We will also use a similar notation for points $p^* \in V^*$. Then, we set $E: L^1(\partial\Omega, |D\chi_\Omega|_\nu) \times L^1(T^{*} \Omega) \rightarrow \mathbb{R}$ by the formula
\begin{equation}\label{Ieq:definitionofE}
E(p_0, \overline{p}) = E_0(p_0) + E_1(\overline{p}), \quad E_0(p_0) = \int_{\partial\Omega} |p_0 - f| \, d|D\1_{\Omega}|_\nu, \quad E_1(\overline{p}) = \| \overline{p} \|_{L^1(T^{*}  \Omega)}.
\end{equation}
We also set $G:W^{1,1}(\Omega,d,\nu) \cap L^2(\Omega,\nu) \rightarrow \mathbb{R}$ by
$$G(u):= \frac12 \int_\Omega u^2 \, d\nu - \int_\Omega ug \, d\nu.$$
The functional $G^* : (W^{1,1}(\Omega,d,\nu) \cap L^2(\Omega,\nu))^* \rightarrow [0,+\infty ]$ is given by
$$G^*(u^*) = \displaystyle\frac12 \int_\Omega (u^*  + g)^2 \, d\nu.$$
Since the operator $A$ is defined by the same formula as in the proof of Theorem \ref{thm:neumann}, conditions \eqref{eq:neumanndiv0} and \eqref{eq:pandpoagree} remain true, so
\begin{equation}\label{eq:dirichletdiv0}
\mbox{div}_0(\overline{p}^*) = - A^* p^*
\end{equation}
and
$$p_0^* = (\overline{p}^* \cdot \nu_\Omega)^- \qquad | D_{\1_{\Omega}} |_\nu-\mbox{a.e. on } \partial\Omega.$$
In particular, $\overline{p}^* \in \mathcal{D}_0^{\infty,2}(\Omega)$. Therefore,
\begin{equation}\label{eq:formulaforgstar2}
G^*(A^*p^*) = \displaystyle\frac12 \int_\Omega ( -\mbox{div}_0(\overline{p}^*)  + g)^2 d\nu.
\end{equation}
Now, let us compute the convex conjugate of the functional $E_0$. The functional $E_0^*: L^\infty(\partial\Omega, |D\chi_\Omega|_\nu) \rightarrow \mathbb{R} \cup \{ \infty \}$ is given by the formula
$$E_0^*(p_0^*) = \left\{ \begin{array}{ll} \displaystyle\int_{\partial\Omega} f \, p_0^* \, d|D\1_{\Omega}|_\nu, \quad &\hbox{if} \ \ |p_0^*| \leq 1, \\[10pt] +\infty,\quad &\hbox{otherwise.}   \end{array}  \right. $$
In fact, we only need to observe that
$$E_0^*(p_0^*) = \sup_{p_0 \in L^1(\partial\Omega, |D\chi_\Omega|_\nu)} \left\{ \int_{\partial \Omega} p_0 \, p_0^* \, d|D\1_{\Omega}|_\nu - \int_{\partial\Omega} |p_0 - f| \, d|D\1_{\Omega}|_\nu \right\}$$
$$= \sup_{p_0 \in L^1(\partial\Omega,|D\chi_\Omega|_\nu)} \left\{ \int_{\partial \Omega} (p_0 - f) \, p_0^* \, d|D\1_{\Omega}|_\nu + \int_{\partial\Omega} f \, p_0^* \, d|D\1_{\Omega}|_\nu  - \int_{\partial\Omega} |p_0 - f| \, d|D\1_{\Omega}|_\nu \right\}.$$
 The functional $E_1^*: L^\infty(T \Omega) \rightarrow [0,\infty]$ was already computed in the proof of Theorem \ref{thm:neumann} and is given by the formula
$$E_1^*(\overline{p}^*) = \left\{ \begin{array}{ll}0, \quad &  \|\overline{p}^*(x) \|_{L^\infty(T\Omega)} \leq 1; \\[10pt] +\infty,\quad &\hbox{otherwise}.\end{array}  \right.$$

 Now, we will apply the Fenchel-Rockafellar duality theorem to the primal problem of the form
\begin{equation}\label{primal1}
\inf_{u \in U} \bigg\{ E(Au) + G(u) \bigg\}
\end{equation}
with $E$ and $G$ defined as above. For $u_0 \equiv 0$ we have $E(Au_0) = \int_{\partial\Omega} |f| \, d|D\1_{\Omega}|_\nu < \infty$, $G(u_0) = 0 < \infty$ and $E$ is continuous at $0$. Then, by the Fenchel-Rockafellar duality theorem we have
\begin{equation}\label{DFR1-TVflow}\inf \eqref{eq:primal} = \sup \eqref{eq:dual}
\end{equation}
and
\begin{equation}\label{DFR2-TVflow} \hbox{the dual problem \eqref{eq:dual} admits at least one solution,}
\end{equation}
 where the dual problem is given by
\begin{equation}\label{dual1}
\sup_{p^* \in L^\infty(\partial\Omega, |D\chi_\Omega|_\nu) \times L^\infty(T  \Omega)} \bigg\{ - E_0^*(-p_0^*) - E_1^*(\overline{p}^*) - G^*(A^* p^*) \bigg\}.
\end{equation}
Keeping in mind the above calculations, we set $\mathcal{Z}$ to be the subset of $V^*$ such that the dual problem does not immediately return $-\infty$, namely
$$
\mathcal{Z} = \bigg\{ p^* \in L^\infty(\partial\Omega, |D\chi_\Omega|_\nu) \times \mathcal{D}_0^{\infty,2}(\Omega): \,  \, \| \overline{p}^* \|_\infty \leq 1; \, \| p_0^* \|_\infty \leq 1; \, p_0^* = (\overline{p}^* \cdot \nu_\Omega)^- \bigg\}.
$$
Hence, we may rewrite the dual problem as
\begin{equation}\label{dual2}
\sup_{p^* \in \mathcal{Z}} \bigg\{ - E_0^*(-p_0^*)- G^*(A^* p^*)\bigg\},
\end{equation}
so finally the dual problem takes the form
\begin{equation}\label{dual3}
\sup_{p^* \in \mathcal{Z}} \bigg\{ - \int_{\partial\Omega} f \, p_0^* \, d|D\1_{\Omega}|_\nu - G^*(A^* p^*)\bigg\}.
\end{equation}
In light of the constraint $p_0^* = (\overline{p}^* \cdot \nu_\Omega)^-$, we may simplify the dual problem a bit. We set
$$
\mathcal{Z}' = \bigg\{ \overline{p}^* \in  \mathcal{D}_0^{\infty,2}(\Omega): \, \, \| \overline{p}^* \|_\infty \leq 1 \bigg\},
$$
note that in light of $p_0^* = (\overline{p}^* \cdot \nu_\Omega)^-$, the constraint $\| p_0^* \|_\infty \leq 1$ is automatically satisfied.  Then, using formula \eqref{eq:formulaforgstar2}, we may rewrite the dual problem as
\begin{equation}\label{dual4}
\sup_{\overline{p}^* \in \mathcal{Z}'} \bigg\{ - \int_{\partial\Omega} f \, (\overline{p}^* \cdot \nu_\Omega)^- \, d|D\1_{\Omega}|_\nu - \displaystyle\frac12 \int_\Omega ( -\mbox{div}_0(\overline{p}^*)  + g)^2 \, d\nu \bigg\}.
\end{equation}

Now, consider the energy functional $\mathcal{G}_f : L^2(\Omega, \nu) \rightarrow (-\infty, + \infty]$ defined by
\begin{equation}\label{11fuctme}
\mathcal{G}_f (v):= \left\{ \begin{array}{ll} \mathcal{TV}_f(v) + G(v),   \quad &\hbox{if} \ v \in BV(\Omega, d, \nu) \cap L^2(\Omega, \nu), \\ \\ + \infty \quad &\hbox{if} \ v \in  L^2(\Omega, \nu) \setminus BV(\Omega, d, \nu).\end{array}\right.
\end{equation}
 As in the proof of Theorem \ref{thm:neumann}, this functional is the extension of the functional $E + G$ (well-defined for functions in $W^{1,1}(\Omega,d,\nu) \cap L^2(\Omega,\nu)$) to the space $BV(\Omega,d,\nu) \cap L^2(\Omega,\nu)$. Since $\mathcal{G}_f$ is coercive, convex and lower semi-continuous, the minimisation problem
$$\min_{v \in L^2(\Omega, \nu)} \mathcal{G}_f(v)$$ admit and optimal solution $u$ and we have
\begin{equation}\label{minimun1}
\min_{v \in L^2(\Omega, \nu)} \mathcal{G}_f(v) = \inf_{v \in U} \bigg\{ E(Av) + G(v) \bigg\}.
\end{equation}
Let us take a sequence $u_n \in W^{1,1}(\Omega,d,\nu)$ which has the same trace as $u$ and converges strictly to $u$ and also $u_n \to u$ in $L^2(\Omega, \nu)$ (given by Lemma \ref{lem:goodapproximation}); then, it is a minimising sequence in \eqref{eq:primal}.  Since we have  \eqref{DFR1-TVflow} and \eqref{DFR2-TVflow}, we may use the $\varepsilon-$subdifferentiability property of minimising sequences given in \eqref{eq:epsilonsubdiff11N} and \eqref{eq:epsilonsubdiff21N}. By equation \eqref{eq:epsilonsubdiff21N}, for every $w \in L^2(\Omega,\nu)$, we have
$$G(w) - G(u_n) \geq \langle (w -u_n), A^* p^* \rangle  -\varepsilon_n.$$
Hence,
$$G(w) - G(u) \geq \langle (w -u), A^* p^* \rangle,$$
and consequently,
$$ A^* p^* \in \partial G(u) = \{ u - g \}. $$
and by \eqref{eq:dirichletdiv0}, we get
\begin{equation}\label{div1}
-\mbox{div}_0(\overline{p}^*) = u -g.
 \end{equation}
On the other hand, equation \eqref{eq:epsilonsubdiff11N} gives
\begin{equation}\label{eq:good}
0 \leq \int_{\partial\Omega} \bigg(|T_\Omega u_n - f| + p_0^* (T_\Omega u_n - f)  \bigg) \, d|D\1_{\Omega}|_\nu + \bigg( \| du_n \|_{ L^1(T^{*} \Omega)} + \int_\Omega du_n(\overline{p}^*) \, d\nu \bigg) \leq \varepsilon_n.
\end{equation}
Because the trace of $u_n$ is fixed (and equal to the trace of $u$), the integral on $\partial\Omega$ does not change with $n$; hence, it has to equal zero. Keeping in mind that $p_0^* = (\overline{p}^* \cdot \nu_\Omega)^-$, we get
\begin{equation}\label{front1}(-\overline{p}^* \cdot \nu_\Omega)^- \in \mbox{sign}(T_\Omega u - f)\quad | D_{\1_{\Omega}} |_\nu-a.e. \ \hbox{on} \ \partial \Omega.
\end{equation}
Since the integral on $\partial\Omega$ in \eqref{eq:good} equals zero and $\| du_n \|_{L^1(T^{*} \Omega)} = \int_{\Omega} |Du_n|_\nu$, in the integral on $\Omega$ we have
\begin{equation}\label{inerior1}
0 \leq \int_\Omega \bigg( |Du_n|_\nu + du_n(\overline{p}^*) \bigg) d\nu \leq \varepsilon_n.
\end{equation}
Finally, keeping in mind that $-\mbox{div}_0(\overline{p}^*) = u-g$ and again using the fact that the trace of $u_n$ is fixed and equal to the trace of $u$, by Gauss-Green's formula we get
$$ \int_\Omega du_n(\overline{p}^*) \, d\nu = - \int_\Omega u_n \, \mbox{div}_0(\overline{p}^*) \, d\nu - \int_{\partial\Omega} (\overline{p}^* \cdot \nu_\Omega)^- \, T_\Omega u_n \, d|D\1_\Omega| $$
$$ =\int_\Omega u_n \, (u -g) \, d\nu -\int_{\partial\Omega} (\overline{p}^* \cdot \nu_\Omega)^- \, T_\Omega u \, d|D\1_\Omega|. $$
Then, applying again Gauss-Green's formula, we have
$$\lim_{n \to \infty}  \int_\Omega du_n(\overline{p}^*) \, d\nu = \int_\Omega u \, (u-g) \, d\nu -\int_{\partial\Omega} (\overline{p}^* \cdot \nu_\Omega)^- \, T_\Omega u \, d|D\1_\Omega| $$
$$ = -  \int_\Omega u \,  \mbox{div}_0(\overline{p}^*) \, d\nu -\int_{\partial\Omega} (\overline{p}^* \cdot \nu_\Omega)^- \, T_\Omega u \, d|D\1_\Omega| = \int_\Omega (\overline{p}^*, Du). $$

Hence, since $u_n$ converges strictly to $u$ and having in mind \eqref{inerior1}, we get
$$ \int_\Omega |Du|_\nu - \int_\Omega (-p^*, Du) = \lim_{n \rightarrow \infty} \bigg( \int_\Omega |Du_n|_\nu - \int_\Omega (-\overline{p}^*, Du) \bigg) = 0.$$
This together with Proposition \ref{prop:boundonAnzellottipairing} implies that
\begin{equation}\label{measure1}
(-\overline{p}^*, Du) = |Du|_\nu \quad \mbox{as measures in } \Omega.
\end{equation}
Then, as consequence of \eqref{div1}, \eqref{measure1} and \eqref{front1}, we have that the pair $(u, -\overline{p}^*)$ satisfies \eqref{1RCond}, \eqref{2RCond} and \eqref{3RCond}, and therefore \eqref{RCond} holds.

To conclude, by \cite[Proposition 2.11]{Brezis} we have
$$ D(\partial \mathcal{TV}_f) \subset  D(\mathcal{TV}_f) =  BV(\Omega,d,\nu) \cap L^2(\Omega,\nu) \subset \overline{D(\mathcal{TV}_f)}^{L^2(\Omega,\nu)} \subset \overline{D(\partial \mathcal{TV}_f)}^{L^2(\Omega,\nu)},$$
from which follows the density of the domain. The proof of the complete accretivity of the operator $\mathcal{A}_f$ is similar to the one given in Theorem \ref{L1Neumann} for the operator $\mathcal{B_N}$.
\end{proof}

As for the Neumann problem, we can also prove the following characterisation of weak solutions in terms of variational inequalities.

\begin{corollary}\label{cor:characterisationweaksolutionsdirichlet}
The following conditions are equivalent: \\
$(a)$ $(u,v) \in \partial\mathcal{TV}_{f}$; \\
$(b)$ $(u,v) \in \mathcal{A}_{f}$, i.e. $u, v \in L^2(\Omega, \nu)$, $u \in BV(\Omega, d, \nu)$ and there exists a vector field  $X \in \mathcal{D}_0^{\infty,2}(\Omega)$ with  $\| X \|_\infty \leq 1$ such that $-\mbox{div}_0(X) = v$ in $\Omega$ and
\begin{equation}\label{eq:interiorconditiondirichlet}
(X,Du) = |Du|_\nu  \quad \hbox{as measures in } \Omega;
\end{equation}
\begin{equation}\label{eq:boundaryconditiondirichlet}
(X \cdot \nu_\Omega)^- \in \mathrm{sign}(T_\Omega u - f) \qquad |D\chi_\Omega|_\nu-\mbox{a.e. on } \Omega;
\end{equation}
$(c)$ $u, v \in L^2(\Omega, \nu)$, $u \in BV(\Omega, d, \nu)$ and there exists a vector field $X \in \mathcal{D}_0^{\infty,2}(\Omega)$ with  $\| X \|_\infty \leq 1$ such that $-\mbox{div}_0(X) = v$ in $\Omega$ and for every $w \in BV(\Omega,d,\nu) \cap L^2(\Omega,\nu)$
\begin{equation}\label{eq:variationalinequalitytvflowdirichlet}
\int_{\Omega} v(w-u) \, d\nu \leq \int_{\Omega} (X,Dw) - \int_{\Omega} |Du|_\nu \qquad\qquad\qquad\qquad\qquad\qquad\qquad
\end{equation}
\begin{equation*}
\qquad\qquad\qquad\qquad\qquad
+ \int_{\partial\Omega} (X \cdot \nu_\Omega)^- \, (T_\Omega w - f) \, d|D\chi_\Omega|_\nu - \int_{\partial\Omega} |T_\Omega u - f| \, d|D\chi_\Omega|_\nu;
\end{equation*}
$(d)$ $u, v \in L^2(\Omega, \nu)$, $u \in BV(\Omega, d, \nu)$ and there exists a vector field  $X \in \mathcal{D}_0^{\infty,2}(\Omega)$ with  $\| X \|_\infty \leq 1$ such that $-\mbox{div}_0(X) = v$ in $\Omega$ and for every $w \in BV(\Omega,d,\nu) \cap L^2(\Omega,\nu)$
\begin{equation}
\int_{\Omega} v(w-u) \, d\nu = \int_{\Omega} (X,Dw) - \int_{\Omega} |Du|_\nu \qquad\qquad\qquad\qquad\qquad\qquad\qquad
\end{equation}
\begin{equation*}
\qquad\qquad\qquad\qquad\qquad
+ \int_{\partial\Omega} (X \cdot \nu_\Omega)^- \, (T_\Omega w - f) \, d|D\chi_\Omega|_\nu - \int_{\partial\Omega} |T_\Omega u - f| \, d|D\chi_\Omega|_\nu;
\end{equation*}
\end{corollary}

\begin{proof}
The equivalence of $(a)$ and $(b)$ is given by Theorem \ref{thm:dirichlet}. To see that $(b)$ implies $(d)$, multiply the equation $v = -\mbox{div}_0(X)$ by $w - u$ and integrate over $\Omega$ with respect to $\nu$. Using the Gauss-Green formula, i.e. Theorem \ref{thm:generalgreensformula}, we get
\begin{equation*}
\int_{\Omega} v (w-u) \, d\nu = - \int_{\Omega} (w - u) \, \mbox{div}_0(X) \, d\nu
\end{equation*}
\begin{equation*}
= \int_{\Omega} (X,Dw) - \int_{\Omega} (X,Du) + \int_{\partial\Omega} (X \cdot \nu_\Omega)^- \, T_\Omega w \, d|D\chi_\Omega|_\nu - \int_{\partial\Omega} (X \cdot \nu_\Omega)^- \, T_\Omega u \, d|D\chi_\Omega|_\nu
\end{equation*}
\begin{equation*}
= \int_{\Omega} (X,Dw) - \int_{\Omega} (X,Du) + \int_{\partial\Omega} (X \cdot \nu_\Omega)^- (T_\Omega w - f) \, d|D\chi_\Omega|_\nu - \int_{\partial\Omega} (X \cdot \nu_\Omega)^- (T_\Omega u - f) \, d|D\chi_\Omega|_\nu
\end{equation*}
\begin{equation*}
= \int_{\Omega} (X,Dw) - \int_{\Omega} |Du|_\nu + \int_{\partial\Omega} (X \cdot \nu_\Omega)^- \, (T_\Omega w - f) \, d|D\chi_\Omega|_\nu - \int_{\partial\Omega} |T_\Omega u - f| \, d|D\chi_\Omega|_\nu,
\end{equation*}
where the last equality follows from \eqref{eq:interiorconditiondirichlet} and \eqref{eq:boundaryconditiondirichlet}.

Clearly, $(d)$ implies $(c)$. To finish the proof, we need to prove that $(c)$ implies $(b)$. If we take $w = u$ in \eqref{eq:variationalinequalitytvflowdirichlet} and reorganise the equation, we get
\begin{equation}\label{eq:variationalequalitydirichlet}
\int_{\Omega} |Du|_\nu + \int_{\partial\Omega} |T_\Omega u - f| \, d|D\chi_\Omega|_\nu \leq \int_{\Omega} (X,Du) + \int_{\partial\Omega} (X \cdot \nu_\Omega)^- \, (T_\Omega u - f) \, d|D\chi_\Omega|_\nu.
\end{equation}
However, by Proposition \ref{prop:boundonAnzellottipairing} and $\| X \|_\infty \leq 1$ we have
\begin{equation}\label{eq:variationalequalitypart1}
\int_\Omega (X,Du) \leq \int_\Omega |Du|_\nu
\end{equation}
and since $\| X \|_\infty \leq 1$, we also have $\| (X \cdot \nu_\Omega)^- \|_\infty \leq 1$, so clearly we have
\begin{equation}\label{eq:variationalequalitypart2}
\int_{\partial\Omega} (X \cdot \nu_\Omega)^- \, (T_\Omega u - f) \, d|D\chi_\Omega|_\nu \leq  \int_{\partial\Omega} |T_\Omega u - f| \, d|D\chi_\Omega|_\nu.
\end{equation}
Hence, in inequality \eqref{eq:variationalequalitydirichlet} we actually have an equality. But this implies that \eqref{eq:variationalequalitypart1} and \eqref{eq:variationalequalitypart2} are also equalities, and this implies \eqref{eq:interiorconditiondirichlet} and \eqref{eq:boundaryconditiondirichlet} respectively.
\end{proof}

\begin{definition}\label{def35}
We define in $L^2(\Omega,\nu)$ the multivalued operator $\Delta^f_{1,\nu}$ by
\begin{center}$(u, v ) \in \Delta^f_{1,\nu}$ \ if, and only if, \ $-v \in \partial\mathcal{TV}_f(u)$.
\end{center}
\noindent  As usual, we will write $v\in \Delta^f_{1,\nu} u$ for $(u,v)\in \Delta^f_{1,\nu}$.
\end{definition}

Our concept of solution of the Dirichlet problem \eqref{Dirichlet100} is the following:
\begin{definition}\label{def:dirichlet1p}
{\rm  Given $u_0 \in L^2(\Omega,\nu)$, we say that $u$ is a {\it weak solution} of the Dirichlet problem \eqref{Dirichlet100} in $[0,T]$, if $u \in  C(0,T; L^2(\Omega))  \cap W^{1,1}(0, T; L^2(\Omega,\nu))$,   $u(0, \cdot) =u_0$, and for almost all $t \in (0,T)$
\begin{equation}\label{def:dirichletflow}
u_t(t, \cdot) \in \Delta^f_{1,\nu}(t, \cdot).
\end{equation}
In other words, there exist vector fields $X(t) \in \mathcal{D}_0^{\infty,2}(\Omega)$ with $\| X(t) \|_\infty \leq 1$ such that for almost all $t \in (0,T)$ the following conditions hold:
$$ \mbox{div}_0(X(t)) = u_t(t, \cdot) \quad \hbox{in} \ \Omega; $$
$$ (X(t), Du(t)) = |Du(t)|_\nu \quad \hbox{as measures};$$
$$ (X(t) \cdot \nu_\Omega)^- \in \mbox{sign}(T_\Omega u(t)-f) \qquad  | D_{\1_{\Omega}} |_\nu-a.e. \ \hbox{on} \ \partial \Omega.$$
}
\end{definition}

Then, as consequence of Theorem \ref{thm:dirichlet}, we have the following existence and uniqueness theorem, where the comparison principle is a consequence of the complete accretivity of the operator $\mathcal{A}_f$.

\begin{theorem}\label{ExistUniqDirichlet} Let $f \in L^1(\partial \Omega, \mathcal{H})$. For any $u_0 \in L^2(\Omega,\nu)$ and $T > 0$ there exists a unique weak solution of the Dirichlet problem  \eqref{Dirichlet100} in $[0,T]$. Moreover, the following comparison principle holds: if $u_1, u_2$ are weak solutions for the initial data $u_{1,0}, u_{2,0} \in  L^2(\Omega, \nu) \cap L^q(\Omega, \nu) $, respectively, then
\begin{equation}\label{DCompPrincipleplaplaceBV}
\Vert (u_1(t) - u_2(t))^+ \Vert_q \leq \Vert ( u_{1,0}- u_{2,0})^+ \Vert_q \quad \hbox{for all} \ 1 \leq q \leq \infty.
\end{equation}
\end{theorem}

Let us also comment on the asymptotic behaviour of the homogenous Dirichlet problem. Assume that $f = 0$, i.e. we study the problem
\begin{equation}\label{ACP12}
\left\{ \begin{array}{ll} u'(t) + \partial \mathcal{TV}_0(u(t)) \ni 0, \quad t \in [0,T] \\[10pt] u(0) = u_0 \end{array}\right.
\end{equation}
with
$$
\mathcal{TV}_0 (u):= \left\{ \begin{array}{ll} \vert Du \vert_\nu (\Omega) + \displaystyle\int_{\partial \Omega} \vert  T_\Omega(u)\vert \, d |D\1_\Omega|_\nu   \quad &\hbox{if} \ u \in BV(\Omega, d, \nu) \cap L^2(\Omega,\nu), \\[10pt]+ \infty \quad &\hbox{if} \ u \in  L^2(\Omega, \nu) \setminus BV(\Omega, d, \nu).\end{array}\right.
$$
We will again apply the results from \cite{BuBu}. Clearly, $\mathcal{TV}_0$ is convex, $1$-homogenous and lower semi-continuous. Moreover,
$$L^2(\Omega, \nu)_0:= \{ u \in L^2(\Omega, \nu) \ : \ \mathcal{TV}_0(u)=0 \}^{\perp} \setminus \{ 0 \} = L^2(\Omega, \nu) \setminus \{ 0 \}.$$
Hence, $\mathcal{TV}_0$ is coercive if and only if there exists a constant $C >0$ such that
\begin{equation}\label{coercD}
\Vert u \Vert_{L^2(\Omega, \nu)} \leq C \mathcal{TV}_0(u), \quad \forall \, u\in L^2(\Omega, \nu) \setminus \{ 0 \},
\end{equation}
which is equivalent to the following {\it Sobolev inequality}
\begin{equation}\label{Sobolev1}
\Vert u  \Vert_{L^2(\Omega, \nu)} \leq C \vert Du \vert_\nu (\Omega) + \int_{\partial \Omega} \vert  T_\Omega(u)\vert \, d |D\1_\Omega|_\nu, \quad \forall \, u\in BV(\Omega, d, \nu) \cap L^2(\Omega,\nu).
\end{equation}
If the above Sobolev inequality holds, we can again apply \cite[Theorem 2.2]{BuBu} and \cite[Theorem 2.3]{BuBu} to get the following result.

\begin{theorem}\label{mainResultDir}
Assume that \eqref{Sobolev1} holds. For  $u_0 \in L^2(\Omega,\nu)$, let  $u(t)$ be the weak solution to the Dirichlet problem \eqref{Dirichlet100}. Then, we have
\begin{itemize}
\item[(i)] (Finite extinction time)
$$T_{\rm ex}(u_0) \leq \frac{\Vert u_0 \Vert_{L^2(\Omega, \nu)}}{\lambda_1(\mathcal{TV}_0)},$$
where
$$T_{\rm ex}(u_0):= \inf \{ T >0 \ : \ u(t) = 0,  \ \ \forall \, t \geq T \}.$$
\item[(ii)] (Asymptotic profiles for finite extinction) Let $$w(t):= \frac{u(t)}{\left(1 - \frac{1}{T_{\rm ex}(u_0)} t \right)},$$
and assume that $w(t)$ converges (possibly up to a subsequence) strongly to some $w_* \in L^2(\Omega, \nu)$ as $t \to T_{\rm ex}(u_0)$. Then it holds $$ \frac{1}{T_{\rm ex}(u_0)} \in \partial \mathcal{TV}_0(w_*),\quad   w_*  \not=0, \quad  \Vert w_* \Vert_{L^2(\Omega, \nu)} \leq \Vert u_0 \Vert_{L^2(\Omega, \nu)}.$$
\item[(iii)] (Ground state as asymptotic profile) An asymptotic profile $w_*$ is a ground state, i.e., $ w_* = {\rm argmin} \frac{\mathcal{TV}_0(w_*)}{\Vert w_* \Vert}$, if and only if $\lim_{t \nearrow T_{\rm ex}}(u_0) = \lambda_1(\mathcal{TV}_0).$
\end{itemize}
\end{theorem}

 Finally, let us comment on the relationship between our results and the only other results (to the best of our knowledge) on the total variation flow in metric measure spaces.

\begin{remark}\label{VariationalSol} {\rm
In \cite{BCC} the authors obtain existence and uniqueness of a variational solution for the Cauchy-Dirichlet problem
\begin{equation}\label{CauchyDirichlet}
\left\{ \begin{array}{ll} u_t  = {\rm div} \left(\frac{ D u}{\vert Du \vert_\nu } \right)  \quad &\hbox{in} \ \ \Omega_T:= \Omega \times (0,T) , \\[10pt] u  = u_0 \quad &\hbox{on} \ \ \partial_{par} \Omega_T:= (\overline{\Omega} \times \{0\}) \cup (\partial \Omega \times (0,T)). \end{array} \right.
\end{equation}
Let $\Omega^*$ a bounded domain that compactly contains $\Omega$. Given $u_0 \in BV(\Omega^*)$, we set
$$BV_{u_0}(\Omega):= \{ u \in BV(\Omega^*) \ : \ u = u_0 \ \ \hbox{a.e. \ on} \ \Omega^* \setminus \Omega \}.$$
Then, $u \in L^1_w(0,T; BV_{u_0}(\Omega)) \cap C([0,T]; L^2(\Omega^*))$ is called a {\it variational solution} of \eqref{CauchyDirichlet} if the following variational inequality
\begin{equation}\label{varineq}
\begin{array}{ll} \displaystyle\int_0^T \vert Du(t) \vert_\nu(\Omega^*) dt \leq \int_0^T \left[\int_{\Omega^*} \partial_t v (v - u) d \nu + \vert Dv(t) \vert_\nu(\Omega^*) \right]dt \\[10pt] - \frac12 \Vert (v-u)(T) \Vert^2_{L^2(\Omega^*} + \frac12 \Vert v(0)-u_0 \Vert^2_{L^2(\Omega^*}  \end{array}
\end{equation}
holds true for any $v \in L^1_w(0,T; BV_{u_0}(\Omega))$ with $\partial_t v \in L^2(\Omega^*_T)$ and $v(0) \in L^2(\Omega^*)$.

Let us see that if $u \in L^1_w(0,T; BV_{u_0}(\Omega)) \cap C([0,T]; L^2(\Omega^*))$ is a weak solution of of the Dirichlet problem  \eqref{Dirichlet100},  with $f \in L^1(\partial\Omega,\mathcal{H})$ such that $T_{\mathbb{X} \setminus \overline{\Omega}} u_0 = f$, then $u$ is a variational solution of  problem \eqref{varineq}.  To this end, we will make a similar calculation as in the proof of Corollary \ref{cor:characterisationweaksolutionsdirichlet}, but we will additionally take into account the different form of the Dirichlet boundary condition and integration with respect to time in the definition of variational solutions. By the definition of weak solutions, there exist vector fields $X(t) \in  \mathcal{D}_0^{\infty,2}(\Omega)$ with $\| X(t) \|_\infty \leq 1$ such that, for almost all $t \in (0,T)$, the following conditions hold:
$$  \mbox{div}_0(X(t)) = \partial_t u(t, \cdot) \quad \hbox{in} \ \Omega; $$
$$ (X(t), Du(t)) = |Du(t)|_\nu \quad \hbox{as measures};$$
$$ (X(t) \cdot \nu_\Omega)^- \in \mbox{sign}(T_\Omega u(t)-f) \qquad  | D_{\1_{\Omega}} |_\nu-a.e. \ \hbox{on} \ \partial \Omega.$$

Then, given a test function $v \in L^1_w(0,T; BV_{u_0}(\Omega))$ with $\partial_t v \in L^2(\Omega^*_T)$ and $v(0) \in L^2(\Omega^*)$,  we want to show that \eqref{varineq} holds. We start by computing the term with the time derivative using the characterisation of weak solutions. Notice that if $u, v \in BV_{u_0}(\Omega)$, then $u - v = 0$ $\nu$-a.e. on $\Omega^* \setminus \Omega$; hence, applying Green formula we have
$$\int_0^T \int_{\Omega^*} \partial_t u (v - u) d \nu dt = \int_0^T \int_{\Omega} \partial_t u (v - u) d \nu dt = \int_0^T \int_{\Omega} \mbox{div}_0(X(t)) (v - u) d \nu dt $$
$$  = - \int_0^T \int_{\Omega} (X(t), Dv(t)) dt + \int_0^T \int_{\Omega} (X(t), Du(t)) dt $$
$$  - \int_{\partial\Omega} T_\Omega v (X(t) \cdot \nu_\Omega)^- d|D\1_\Omega|_\nu + \int_{\partial\Omega} T_\Omega u (X(t) \cdot \nu_\Omega)^- d|D\1_\Omega|_\nu$$
$$  = - \int_0^T \int_{\Omega} (X(t), Dv(t)) dt + \int_0^T \int_{\Omega} (X(t), Du(t)) dt $$
$$ - \int_{\partial\Omega} (T_\Omega v - f) (X(t) \cdot \nu_\Omega)^- d|D\1_\Omega|_\nu + \int_{\partial\Omega} (T_\Omega u - f) (X(t) \cdot \nu_\Omega)^- d|D\1_\Omega|_\nu$$

and
$$\int_0^T \int_{\Omega^*} (\partial_t v - \partial_t u )(v - u) d \nu dt = \frac12 \Vert (v-u)(T) \Vert^2_{L^2(\Omega^*,\nu)} - \frac12 \Vert v(0)-u_0 \Vert^2_{L^2(\Omega^*,\nu)}. $$
 Since $u_0$ is a weak solution and $T_{\mathbb{X} \setminus \overline{\Omega}} u_0 = f$, adding the two equalities we get
$$\int_0^T \int_{\Omega^*} \partial_t v (v - u) d \nu dt $$
$$= - \int_0^T \int_{\Omega} (X(t), Dv(t)) dt + \int_0^T \int_{ \Omega} (X(t), Du(t)) dt - \int_{\partial\Omega} (T_\Omega v - f) (X(t) \cdot \nu_\Omega)^- d|D\1_\Omega|_\nu $$
$$  + \int_{\partial\Omega} (T_\Omega u - f) (X(t) \cdot \nu_\Omega)^- d|D\1_\Omega|_\nu + \frac12 \Vert (v-u)(T) \Vert^2_{L^2(\Omega^*,\nu)} - \frac12 \Vert v(0)-u_0 \Vert^2_{L^2(\Omega^*,\nu)} $$
$$\geq - \int_0^T \vert Dv(t)\vert_\nu ( \Omega) dt + \int_0^T \vert Du(t)\vert_\nu ( \Omega) dt - \int_0^T \int_{\partial\Omega} |T_\Omega v - f| d|D\1_\Omega| dt $$
$$  + \int_0^T \int_{\partial\Omega} |T_\Omega u - f| d|D\1_\Omega| dt - \int_0^T |Du_0|(\Omega^* \setminus \overline{\Omega}) dt + \int_0^T |Du_0|(\Omega^* \setminus \overline{\Omega}) dt
$$
$$ + \frac12 \Vert (v-u)(T) \Vert^2_{L^2(\Omega^*,\nu)} - \frac12 \Vert v(0)-u_0 \Vert^2_{L^2(\Omega^*,\nu)}$$
$$= - \int_0^T \vert Dv(t)\vert_\nu (\Omega^*) dt + \int_0^T \vert Du(t)\vert_\nu (\Omega^*) dt + \frac12 \Vert (v-u)(T) \Vert^2_{L^2(\Omega^*,\nu)} - \frac12 \Vert v(0)-u_0 \Vert^2_{L^2(\Omega^*,\nu)}.$$
In the final equality we used \cite[Proposition 5.11]{HKLL}. Then, \eqref{varineq} holds and consequently $u$ is a variational solution,  so the definition of weak solutions introduced in this paper is consistent with the definition of variational solutions introduced in \cite{BCC}. Introducing weak solutions requires more assumption on the set $\Omega$, but allows us to describe the solutions more precisely. $\blacksquare$}
\end{remark}

\noindent {\bf Acknowledgment.} The first author has been partially supported by the DFG-FWF project FR 4083/3-1/I4354, by the OeAD-WTZ project CZ 01/2021, and by the project 2017/27/N/ST1/02418 funded by the National Science Centre, Poland. The second author has been partially supported by the Spanish MCIU and FEDER, project PGC2018-094775-B-100.

\end{document}